\documentclass{article}
\usepackage{amssymb}
\usepackage{amsmath}
\usepackage{amsfonts}

\newtheorem{theorem}{Theorem}
\newtheorem{acknowledgement}[theorem]{Acknowledgement}

\newtheorem{lemma}[theorem]{Lemma}

\newtheorem{remark}{Remark}

\newenvironment{proof}[1][Proof]{\noindent\textbf{#1.} }{\ \rule{0.5em}{0.5em}}
\input{tcilatex}
\begin{document}

\title{Block Thresholding on the Sphere }
\author{Claudio Durastanti\thanks{%
Research supported by ERC Grant n. 277742 Pascal} \thanks{%
E-mail: durastan@mat.uniroma2.it} \\
University of "Tor Vergata", Rome}
\maketitle

\begin{abstract}
The aim of this paper is to study nonparametric regression estimators on the
sphere based on needlet block thresholding. The block thresholding procedure
proposed here follows the method introduced by Hall, Kerkyacharian and
Picard in \cite{hkp}, \cite{hkp2}, which we modifyto exploit the properties
of spherical needlets. We establish convergence rates, and we show that they
attain adaptivity over Besov balls in the regular region. This work is
strongly motivated by issues arising in Cosmology and Astrophysics,
concerning in particular the analysis of Cosmic rays.


\end{abstract}

\section{\protect\bigskip Introduction}

Over the last years, wavelet techniques have been used to achieve remarkable
results in the field of statistics, in particular in the framework of
minimax estimation in nonparametric settings. The pioneering work in this
area was provided by Donoho et al. in \cite{donoho1}, where authors proved
that nonlinear wavelet estimators based on thresholding techniques attain
nearly optimal minimax rates, up to logarithmic terms, for a large class of
unknown density and regression functions. Since then, this research area has
been deeply investigated and extended - we suggest for instance \cite{WASA}
as a textbook reference.\ In this paper, we shall focus on block
thresholding procedure; loosely speaking, this method keeps or annihilates
blocks of wavelet coefficients on each given level (for more details, see 
\cite{WASA}), hence representing an intermediate way between local and
global thresholding, which fix a threshold respectively for each coefficient
and for all of them. Block thresholding was initially suggested in \cite%
{efroim} for orthogonal series estimators and later applied by \cite{hkp}
for both wavelet and kernel density estimation on $\mathbb{R}$ (see also 
\cite{hkp2}); it was also used in \cite{caibloc} in the framework of Oracle
inequalities, while overlapping block thresholding estimators were studied
in \cite{caisilver}. Block thresholding was also applied to study adaptivity
in density estimation in \cite{chickencai}, a data-driven block thresholding
procedure for wavelet regression is investigated in \cite{caizhou}, while
wavelet-based block thresholding rules on maxisets are proposed by \cite%
{autin}. \newline
A huge number of results concerns estimation within the thresholding
paradigm in standard Euclidean frameworks, such as $\mathbb{R}$ or $\mathbb{R%
}^{d}$; more recently general settings, such as spherical data or more
general manifolds have been considered. Here we focus on a second-generation
wavelet system on the sphere, the so-called needlets. Needlets were
introduced by Narcowich, Petrushev and Ward in \cite{npw1}, \cite{npw2};
their stochastic properties, when exploited on spherical random fields, were
studied in \cite{bkmpAoS}, \cite{bkmpBer}, \cite{ejslan} and \cite{spalan}.
This approach has been extended to more general manifolds by \cite{gm1}, 
\cite{gm2}, \cite{gm3}, while their generalization to spin fiber bundles on
the sphere were described in \cite{gelmar}, \cite{gelmar2010}. Most of these
researches can be motivated by applications to Cosmology and Astrophysics:
for instance, a huge amount of spherical data, concerning the Cosmic
Microwave Background radiation, are being provided by satellite missions
WMAP and Planck, see \cite{pbm06}, \cite{mpbb08}, \cite{pietrobon1}, \cite%
{fay08}, \cite{pietrobon2}, \cite{rudjord1}, \cite{dela08}, \cite{rudjord2}, 
\cite{dlm1} and \cite{dlm2} for more details. The applications mentioned
here, however, do not concern thresholding estimation, but rather they can
be related to the study of random fields on the sphere, such as angular
power spectrum estimation, higher-order spectra, testing for Gaussianity and
isotropy, and several others (see also \cite{cama}). Of more direct interest
here are experiments concerning incoming directions of Ultra High Energy
Cosmic Rays, such as the AUGER Observatory (http://www.auger.org).\
Ultra-High Energy Cosmic Rays are particles with energy above $10^{18}$ eV
reaching the Earth. Even if they were discovered almost a century ago, their
origin, their mechanisms of acceleration and propagation are still unknown.
As described in \cite{bkmpAoSb}, see also \cite{faytest}, an efficient
nonparametric estimation of the density function of these data would explain
the origin of the High Energy Cosmic Rays, i.e. if it is uniform, they are
generated by cosmological effects, such as the decay of the massive
particles generated during the Big Bang, or, on the other hand, if it is
highly non-uniform and, moreover, strongly correlated with the local
distribution of nearby Galaxies, it implies that the they are generated by
astrophysical phenomena, as for instance the acceleration into Active
Galactic Nuclei. Massive amount of data in this area are expected to be
available in the next few years. Also in view of this application, the
needlet approach was recently applied within the thresholding paradigm to
the estimation of the directional data: the seminal contribution in this
field is due to \cite{bkmpAoSb}, see also \cite{Kerkypicard}, \cite{knp},
while applications to astrophysical data is still under way, see for
instance \cite{fay08},\ \cite{faytest} and \cite{Iuppa} (the latter related
to Gamma Rays, another major field where these ideas have proved extremely
fruitful). Minimax estimators for spherical data, outside the needlets
approach, were also studied by Kim and coauthors (see \cite{kim}, \cite%
{kimkoo}, \cite{kookim}). Furthermore, adaptive nonparametric regression
estimators of spin-functions, based on spin pure and mixed needlets defined
in \cite{gelmar}, \cite{gelmar2010}, were investigated in \cite{dgm}. In
this case, the needlet nonparametric regression estimators were built on
spin fiber bundles on the sphere, i. e. the function to be estimated does
not take as its values scalars but algebraic curves living on the tangent
plane for each point of the sphere.

This work aims to extend the results established in \cite{bkmpAoSb} and \cite%
{dgm} towards the needlet block thresholding procedure following two main
directions. First of all, we will suggest a construction of blocks of
needlet coefficients, exploiting the Voronoi cells based on the geodesic
distance on the sphere. Then, we will define the needlet block thresholding
estimator, whose we will achieve a near optimal convergence rate. In view of
this purpose, we will use both the needlet properties established in \cite%
{npw1}, \cite{npw2} (see also \cite{marpecbook}) and a set of well
consolidated standard techniques, introduced by \cite{donoho1} (see also 
\cite{WASA}), remarking that this kind of approach has been also applied
within the needlet framework, just considering local thresholding, by \cite%
{bkmpAoSb} and \cite{dgm}. We also remark that we will describe the
nonparametric regression problem in terms of the so-called Gaussian white
noise model, able to give suitable approximation of discrete nonparametric
regression model, already commonly used in problems over $\mathbb{R}$ (see
for instance \cite{tsybakov} and Section \ref{sec:block}) and here used over
the $d$-dimensional sphere for the first time, at least at our knowledge.

Indeed, consider $f\in L^{p}\left( S^{d}\right) $, the needlet frame $%
\left\{ \psi _{jk}\right\} _{j,k}$, whose main properties will be described
in Section \ref{sec:need}, and the corresponding needlet coefficients $%
\left\{ \beta _{jk}\right\} _{j,k}$ given as 
\begin{equation*}
\beta _{jk}:=\int_{\mathbb{S}^{d}}f\left( x\right) \psi _{jk}\left( x\right)
dx\text{ .}
\end{equation*}%
As shown in Section \ref{sec:need}, from the reconstruction formula (\ref%
{reconst}), we can describe $f$ in terms of needlet decomposition as 
\begin{equation*}
f\left( x\right) =\sum_{j=0}^{+\infty }\sum_{k=1}^{N_{j}}\beta _{jk}\psi
_{jk}\left( x\right) \text{ ,}
\end{equation*}%
where the equality holds in the $L^{2}$-sense. Consider now $X_{n}$, a
sample path of an isonormal Gaussian process with common mean $f$ (see
Section \ref{sec:block}), equivalent to the available dataset, where the
random element $X_{n}\left( \psi _{jk}\right) $ can be described as 
\begin{equation*}
X_{n}\left( \psi _{jk}\right) =\beta _{jk}+\varepsilon _{jk;n}=:\widehat{%
\beta }_{jk}\text{ ,}
\end{equation*}%
so that $\varepsilon _{jk;n}$ is the noise with the properties described in
Section \ref{sec:block}. For any given resolution level $j$, we therefore
build $S_{j}$ blocks, labeled as $R_{j;s}$, $s=1,...,S_{j}$, each of them
containing $\ell _{j}$ cubature points. We define 
\begin{equation*}
\widehat{A}_{js;p}:=\frac{1}{\ell _{j}}\sum_{k\in R_{j;s}}\widehat{\beta }%
_{jk}^{p}\text{ ,}
\end{equation*}%
and the corresponding weight function%
\begin{equation*}
w_{js;p}:=I\left( \left\vert \widehat{A}_{js;p}\right\vert >\kappa
t_{n}^{p}\right) \,\ ,
\end{equation*}%
(more details on $\kappa $ and $t_{n}$ are in Section \ref{sec:block}).
Hence we build the needlet block thresholding estimator for $f$%
\begin{equation*}
\widehat{f}=\sum_{j=0}^{J_{n}}\sum_{s=0}^{S_{j}}w_{js;p}\left( \sum_{k\in
R_{j;s}}\widehat{\beta }_{jk}\psi _{jk}\right) \text{ .}
\end{equation*}%
We will show that, under some regularity conditions (cfr. Theorem \ref%
{maintheorem}, Section \ref{sec:minimax}), there exists $c_{p}>0$ so that 
\begin{equation*}
\sup_{f\in \mathcal{B}_{\pi q}^{r}\left( M\right) }\mathbb{E}\left\Vert 
\widehat{f}-f\right\Vert _{L^{p}\left( \mathbb{S}^{d}\right) }^{p}\leq
c_{p}n^{-\alpha \left( r,\pi ,p\right) }\text{ ,}
\end{equation*}%
where $\alpha \left( r,\pi ,p\right) $ corresponds to the optimal rate in
the regular zone (for definition, see Section \ref{sec:minimax}) and it
attains almost the optimal rate in the sparse zone (recall that in the soft
thresholding procedure the minimax rate is $\left( n/\log n\right) ^{-\alpha
\left( r,\pi ,p\right) })$. The improvement achieved by block thresholding
in the regular zone can be explained by the better trade-off between bias
and variance; the latter is due to the information in nearby coefficients.
Note that adaptivity is conditional upon a very careful choice of the block
sizes (see Section \ref{sec:block} and also \cite{caibloc2}, \cite{hkp} and 
\cite{hkp2}). For what concerns the sparse zone, the choice of the block
size itself will lead us to a not optimal result, as motivated in Sections %
\ref{sec:block} and \ref{sec:aux}.

The plan of the paper is as follows: Section \ref{sec:need} will recall some
preliminary notions, as needlets, their main properties and the Besov
spaces. Section \ref{sec:block} will describe the block thresholding
procedure we build for needlet regression estimation , while Section \ref%
{sec:minimax} will present the main minimax results. Section \ref{sec:aux}
will collect some auxiliary probabilistic results, while Section \ref%
{sec:proof} will exploit the proof of the main result of this work, named as
Theorem \ref{maintheorem}. Finally, Section \ref{sec:conclu} will compare
our results with the others in literature concerning needlet thresholding.

\section{Background results \label{sec:need}}

In this Section, we will review briefly a few of well-known features about
the Voronoi cells on the sphere, the spherical needlet construction and the
Besov spaces.

For what concerns the definition of Voronoi cells, we are following strictly 
\cite{bkmpBer}: further details can be found for instance in the textbook 
\cite{marpecbook}, see also \cite{bkmpAoS} and \cite{npw2}. From now on,
given two positive sequences $\left\{ a_{j}\right\} $ and $\left\{
b_{j}\right\} $, we write that $a_{j}\approx b_{j}$ if there exists a
constant $c>1$ so that $c^{-1}a_{j}\leq b_{j}\leq ca_{j}$ for all $j$. Let
us call $\mathbb{S}^{d}$ the unit sphere of $\mathbb{R}^{d+1}$. Furthermore, 
$B_{x_{0}}\left( \alpha \right) =\left\{ x\in \mathbb{S}^{d}:d\left(
x,x_{0}\right) <\alpha \right\} $, where $d\left( \cdot ,\cdot \right) $ is
the natural geodesic distance over the sphere, denotes the standard open
ball on $\mathbb{S}^{d}$ around $x_{0}\in \mathbb{S}^{d}$, while $\left\vert
A\right\vert $ is the surface measure of a general subset $A\subset \mathbb{S%
}^{d}$: let us recall that this is the unique positive measure invariant by
rotation, with total mass $\omega _{d}=\left( 2\pi \right) ^{\left(
d+1\right) /2}/\Gamma \left( \left( d+1\right) /2\right) $. Given $%
\varepsilon >0$, the set $\Xi _{\varepsilon }=\left\{
x_{1},...,x_{N}\right\} $ of points on $\mathbb{S}^{d}$, such that for $%
i\neq j$ we have $d\left( x_{i},x_{j}\right) >\epsilon $, is called a \emph{%
maximal }$\varepsilon $\emph{-net} if it satisfies $d\left( x,\Xi
_{\varepsilon }\right) <\varepsilon $ for $x\in \mathbb{S}^{d}$, $\cup
_{x_{i}\in \Xi _{\varepsilon }}B_{x_{i}}\left( \varepsilon \right) =$ $%
\mathbb{S}^{d}$ and $B_{x_{i}}\left( \varepsilon /2\right) \cap
B_{x_{j}}\left( \varepsilon /2\right) =\varnothing $, for $i\neq j$. For all 
$x_{i}\in \Xi _{\varepsilon }\,$, a family of Voronoi cells is defined as%
\begin{equation}
\mathcal{V}\left( x_{i}\right) =\left\{ x\in \mathbb{S}^{d}:\text{for }j\neq
i,\text{ }d\left( x,x_{i}\right) <d\left( x,x_{j}\right) \right\} \text{.}
\label{voronoi}
\end{equation}%
In \cite{bkmpBer} it is proved that:%
\begin{equation*}
B_{x_{i}}\left( \frac{\varepsilon }{2}\right) \subset \mathcal{V}\left(
x_{i}\right) \subset B_{x_{i}}\left( \varepsilon \right) \text{ .}
\end{equation*}

Now, we resume the construction of the scalar needlet framework, suggesting
for a more detailed discussion \cite{npw1}, \cite{npw2}, see also \cite%
{bkmpAoSb} and \cite{marpecbook}. A needlet system describes a
well-localized tight frame on the sphere: it is a well-known fact (cfr. \cite%
{npw1}) that any function belonging to $L^{2}\left( \mathbb{S}^{d}\right) $
can be represented as a linear combination of the components of that frame,
preserving some of the most relevant properties of needlets. Indeed, let us
recall that the space $L^{2}\left( \mathbb{S}^{d}\right) $ of
square-integrable functions on the sphere can be decomposed as the direct
sum of the spaces $H_{l}$ of harmonic polynomials of degree $l$, spanned by
spherical harmonics of degree $l$, whose definition and properties can be
found in \cite{vmk} and \cite{bkmpAoSb}; here we just recall that its
dimension corresponds to $g_{l,d}=\binom{l+d}{d}-\binom{l+d-2}{d}$. For
every $f\in L^{2}\left( \mathbb{S}^{d}\right) $, the following kernel
operator describes the orthogonal projector onto $H_{l}$:%
\begin{equation*}
P_{H_{l}}f\left( x\right) =\int_{\mathbb{S}^{d}}L_{l}\left( \left\langle
x,y\right\rangle \right) f\left( y\right) dy\text{ ,}
\end{equation*}%
where $L_{l}$ is the Gegenbauer polyomial with parameter $\left( d-1\right)
/2$ and degree $l$, normalized so that 
\begin{equation*}
\int_{\mathbb{S}^{d}}L_{l}\left( x\right) L_{m}\left( x\right) \left(
1-x^{2}\right) ^{\frac{d}{2}-1}dx=\frac{g_{l,d}\Gamma \left( \frac{d}{2}%
\right) ^{2}}{\Gamma \left( d\right) \omega _{d}^{2}}\delta _{l,k}\text{ .}
\end{equation*}%
Following \cite{npw1}, \cite{npw2} (see also \cite{bkmpAoSb}), if we
consider 
\begin{equation*}
\Pi _{l}=\underset{l^{\prime }=0}{\overset{l}{\bigoplus }}H_{l^{^{\prime }}}%
\text{,}
\end{equation*}%
the space of the restrictions to $\mathbb{S}^{d}$ of the polynomials of
degree less (and equal) to $l$, the following quadrature formula holds (see
for instance \cite{bkmpAoSb}): given $l\in 
\mathbb{N}
$, there exists a finite subset $\chi _{l}$ such that a positive real number 
$\lambda _{\xi }$ (the cubature weight) corresponds to each $\xi \in \chi
_{l}$ (the cubature point) and for all $f\in $ $\Pi _{l}$, 
\begin{equation*}
\int_{\mathbb{S}^{d}}f\left( x\right) dx=\underset{\xi \in \chi _{l}}{\sum }%
\lambda _{\xi }f\left( \xi \right) \text{.}
\end{equation*}

Given $B>1\,$\ and a resolution level $j$, we call $\chi _{\left[ B^{2\left(
j+1\right) }\right] }=\mathcal{Z}_{j}$, $card\left( \mathcal{Z}_{j}\right)
=N_{j}$; since now any element of the set of cubature points and weights, $%
\left\{ \xi _{jk},\lambda _{jk}\right\} $, will be indexed by $j$, the
resolution level, and $k$, the cardinality over $j$, belonging to $\mathcal{Z%
}_{j}$. Furthermore, we choose $\left\{ \mathcal{Z}_{j}\right\} _{j\geq 1}$%
to be nested so that$\ $%
\begin{equation}
N_{j}\approx B^{dj},\lambda _{jk}\approx B^{-dj}\text{ .}  \label{lambdaeN}
\end{equation}%
We consider a symmetric, real-valued, nonnegative function $b\left( \cdot
\right) $ (see again \cite{bkmpAoSb}) such that

\begin{enumerate}
\item it has compact support on $\left[ B^{-1},B\right] $;

\item $b\in C^{\infty }\left( 
\mathbb{R}
\right) $;

\item the following \textit{unitary property }holds for \textit{\ }$%
\left\vert \xi \right\vert \geq 1$:
\end{enumerate}

\begin{equation*}
\underset{j\geq 0}{\sum }b^{2}\left( \frac{\xi }{B^{j}}\right) =1\text{ .}
\end{equation*}%
For each $\xi _{jk}\in \mathcal{Z}_{j}$, given $b\left( \cdot \right) $ and $%
B$, the scalar needlets are defined as:%
\begin{equation*}
\psi _{jk}\left( x\right) =\sqrt{\lambda _{jk}}\underset{B^{j-1}<l<B^{j+1}}{%
\sum }b\left( \frac{l}{B^{j}}\right) L_{l}\left( \left\langle x,\xi
_{jk}\right\rangle \right) \text{ .}
\end{equation*}%
The properties of the function $b\left( \cdot \right) $ yield to three basic
properties of the needlets. Indeed, from the infinite differentiability of $%
b\left( \cdot \right) $, we obtain a quasi-exponential localization
property\ (see for instance \cite{npw2}), which states that for $k\in 
\mathbb{N}$, there exists $c_{k}$ such that for $x\in \mathbb{S}^{d}$%
\begin{equation}
\left\vert \psi _{jk}\left( x\right) \right\vert \leq \frac{c_{k}B^{\frac{d}{%
2}j}}{\left( 1+B^{\frac{d}{2}j}d\left( \xi _{jk},x\right) \right) ^{k}}\text{%
,}  \label{localization}
\end{equation}%
where $d\left( \xi _{jk},x\right) $ is the geodesic distance on the sphere.
In view of this property, it is possible to fix a bound (upper and lower),
for the norms of needlets on $L^{p}\left( \mathbb{S}^{d}\right) $, for $%
1\leq p\leq +\infty $. Given $p$, there exist two positive constants $c_{p}$
and $C_{p}$ such that 
\begin{equation}
c_{p}B^{jd\left( \frac{1}{2}-\frac{1}{p}\right) }\leq \left\Vert \psi
_{jk}\right\Vert _{L^{p}\left( \mathbb{S}^{d}\right) }\leq C_{p}B^{jd\left( 
\frac{1}{2}-\frac{1}{p}\right) }\text{ .}  \label{boundnorm}
\end{equation}%
Because the function $b\left( \cdot \right) $ has compact support in $\left[
B^{-1},B\right] $, it follows that $b\left( \frac{l}{B^{j}}\right) $ has
compact support in $\left[ B^{j-1},B^{j+1}\right] $, hence needlets have
compact support in the harmonic domain. Finally, the unitary property leads
to the following reconstruction formula (see again \cite{npw1}): for $f\in
L^{2}\left( \mathbb{S}^{d}\right) $, in the $L^{2}$ sense, 
\begin{equation}
f(x)=\sum_{j,k}\beta _{jk}\psi _{jk}(x)\text{ ,}  \label{reconst}
\end{equation}%
\begin{equation}
\beta _{jk}:=\left\langle f,\psi _{jk}\right\rangle _{L^{2}\left( \mathbb{S}%
^{d}\right) }=\int_{\mathbb{S}^{d}}\overline{\psi }_{jk}\left( x\right)
f\left( x\right) dx\text{ ,}  \label{needcoeffic}
\end{equation}%
where $\beta _{jk}$ are the so-called needlet coefficients.

Before concluding this Section, we recall the definition and some main
properties of the Besov spaces, referring again to \cite{bkmpAoSb}, \cite%
{dgm} and \cite{WASA} for further theoretical details and discussions. Let $%
f $ $\in L^{\pi }\left( \mathbb{S}^{d}\right) $; we define 
\begin{equation*}
G_{k}\left( f,\pi \right) =\inf_{H\in \mathcal{H}_{k}}\left\Vert
f-H\right\Vert _{L^{\pi }\left( \mathbb{S}^{d}\right) }\text{ ,}
\end{equation*}%
which is the approximation error when replacing $f$ \ by an element in $%
\mathcal{H}_{k}$ . The Besov space $\mathcal{B}_{\pi q}^{r}$ is therefore
defined as the space of functions such that $f\in L^{\pi }\left( \mathbb{S}%
^{d}\right) $ and%
\begin{equation*}
\left( \sum_{k=0}^{\infty }\frac{1}{k}\left( k^{r}G_{k}\left( f,\pi \right)
\right) ^{q}\right) <\infty \text{ .}
\end{equation*}%
The last condition is equivalent to%
\begin{equation*}
\left( \sum_{j=0}^{\infty }\left( B^{jr}G_{B^{j}}\left( f,\pi \right)
\right) ^{q}\right) <\infty \text{ .}
\end{equation*}%
Moreover, $F\in \mathcal{B}_{\pi q}^{r}$ if and only if, for every $%
j=1,2,\ldots $%
\begin{equation*}
\left( \sum_{k}\left( \left\vert \beta _{jk}\right\vert \left\Vert \psi
_{jk}\right\Vert _{L^{\pi }\left( \mathbb{S}^{d}\right) }\right) ^{\pi
}\right) ^{\frac{1}{\pi }}=\varepsilon _{j}B^{-jr}
\end{equation*}%
where $\varepsilon _{j}\in \ell _{q}$ and $B>1$. The Besov norm is defined
as follows:%
\begin{equation*}
\left\Vert f\right\Vert _{\mathcal{B}_{\pi q}^{r}}=\left\{ 
\begin{matrix}
\left\Vert f\right\Vert _{L^{\pi }\left( \mathbb{S}^{d}\right) }+\left[
\sum_{j}B^{jq\left( r+d\left( \frac{1}{2}-\frac{1}{\pi }\right) \right)
}\left\{ \sum_{k}\left\vert \beta _{jk}\right\vert ^{\pi }\right\} ^{\frac{q%
}{\pi }}\right] ^{\frac{1}{q}} & \text{ \ }q<\infty \\ 
\left\Vert f\right\Vert _{L^{\pi }\left( \mathbb{S}^{d}\right) }+\underset{j}%
{\sup }B^{j\left( r+d\left( \frac{1}{2}-\frac{1}{\pi }\right) \right)
}\left\Vert \left( \beta _{jk}\right) _{k}\right\Vert {_{\ell _{\pi }}} & 
\text{ \ }q=\infty%
\end{matrix}%
\right. \text{.}
\end{equation*}%
As shown for instance in \cite{bkmpAoSb}, if $\max \left( 0,1/\pi
-1/q\right) <r$ and $\pi ,q>1$, then we have%
\begin{equation*}
f\in \mathcal{B}_{\pi q}^{r}\Leftrightarrow \left\Vert f\right\Vert _{%
\mathcal{B}_{\pi q}^{r}}<\infty \text{ .}
\end{equation*}%
The Besov spaces present, among their properties, some embeddings which will
be pivotal in our proofs below. As proven in \cite{bkmpAoSb} and \cite{dgm},
we have that, for $\pi _{1}\leq \pi _{2},$ $q_{1}\leq q_{2}$%
\begin{equation}
\mathcal{B}_{\pi q_{1}}^{r}\subset \mathcal{B}_{\pi q_{2}}^{r}\text{ },\text{
}\mathcal{B}_{\pi _{2}q}^{r}\subset \mathcal{B}_{\pi _{1}q}^{r}\text{ , }%
\mathcal{B}_{\pi _{1}q}^{r}\subset \mathcal{B}_{\pi _{2}q}^{r-d\left( \frac{1%
}{\pi _{1}}-\frac{1}{\pi _{2}}\right) }.  \label{embeddings}
\end{equation}

\section{Needlet Block Thresholding on the Sphere\label{sec:block}}

In this Section we will discuss the needlet estimators for nonparametric
regression problems and, then, we will suggest a procedure to fix blocks for
any given resolution level $j$ and, consequently, we will define the
so-called needlet block threshold estimator. The construction of the needlet
estimators is close to the one described in \cite{bkmpAoSb}, \cite{dgm} for
local thresholding, in turn an adaptation to the sphere of the procedure
developed on $\mathbb{R}$ in \cite{hkp}, \cite{hkp2}, see also \cite{WASA}.

We start by introducing the Gaussian white noise model over the line
segment, then we will extend it to the $d$-dimensional sphere using the
so-called uncentered isonormal Gaussian processes.

Usually, in the mathematical statistics literature, a nonparametric
regression problem over the line segment $\left[ 0,1\right] $ is defined by
the following Gaussian white noise model, e.g., the stochastic differential
equation (see for instance \cite{tsybakov})%
\begin{equation}
dY_{t}=f\left( t\right) dt+\varepsilon dW\left( t\right) \text{ , }t\in %
\left[ 0,1\right] \text{ ,}  \label{whitenoise}
\end{equation}%
where $W$ \ is a standard Wiener process on $\left[ 0,1\right] $, $f$ is an
unknown function over $\left[ 0,1\right] $ and $\varepsilon =n^{-1/2}$, for $%
n$ a growing sequence of integers. It is assumed that a sample path $%
X=\left\{ Y\left( t\right) ,0\leq t\leq 1\right\} $ is observed; the
statistical problem regards the estimation of the unknown function $f\in 
\mathcal{F}$, where $\mathcal{F}$ is a given nonparametric class of
functions. For $x\in \left[ 0,1\right] $, the function $x\mapsto f_{n}\left(
x,X\right) $, defined on $\left[ 0,1\right] $ and measurable with respect $X$%
, is the estimator of $f$.

\begin{remark}
\label{whitecorresp}As proven in \cite{brownlow}, see also for instance \cite%
{tsybakov} Section 1.10 and the references therein, the nonparametric linear
regression model and linear regression in terms of the Gaussian white noise
model are asymptotically equivalent. Indeed, consider the process $Y$ in (%
\ref{whitenoise}). If $\Delta >0$, we have%
\begin{equation*}
\frac{Y\left( t+\Delta \right) -Y\left( t\right) }{\Delta }=\frac{1}{\Delta }%
\int_{t}^{t+\Delta }f\left( s\right) ds+\frac{\varepsilon }{\Delta }\left(
W\left( t+\Delta \right) -W\left( t\right) \right) \text{.}
\end{equation*}%
Now, if 
\begin{equation*}
y\left( t\right) =\frac{Y\left( t+\Delta \right) -Y\left( t\right) }{\Delta }%
\text{, }z\left( t\right) =\frac{\varepsilon }{\Delta }\left( W\left(
t+\Delta \right) -W\left( t\right) \right) \text{ , }
\end{equation*}%
for any $t\in \left[ 0,1\right] $, $z\left( t\right) $ is a centered
Gaussian with variance $\varepsilon ^{2}/\Delta $. Taking $\varepsilon =1/%
\sqrt{n}$ and $\Delta =1/n$, $z\left( t\right) \sim \mathcal{N}\left(
0,1\right) $. Up to deterministic residuals, for sufficient small $\Delta $
and sufficient smooth $f$, 
\begin{equation*}
\frac{1}{\Delta }\int_{t}^{t+\Delta }f\left( s\right) ds-f\left( t\right)
\rightarrow 0\text{ ,}
\end{equation*}%
hence 
\begin{equation*}
y\left( t\right) =f\left( t\right) +z\left( t\right) \text{ .}
\end{equation*}%
For $i=1,...,n$, we take $X_{i}=i/n$, $Y_{i}=Y\left( X_{i}\right) ,$ $%
z_{i}=z\left( X_{i}\right) $, so that 
\begin{equation*}
Y_{i}=f\left( X_{i}\right) +z_{i}\text{ ,}
\end{equation*}%
which corresponds to the nonparametric regression model with regular design
and i.i.d. errors $z_{i}$ distributed as $\mathcal{N}\left( 0,1\right) $.
Further details can be found in \cite{tsybakov}.
\end{remark}

If we consider the $d$-dimensional sphere, we can describe the same problem
in terms of the so-called uncentered isonormal Gaussian processes with mean $%
f$. Following \cite{nourdinpeccati}, an isonormal Gaussian process over $%
\mathfrak{H}$ is defined as $X=\left\{ X\left( h\right) :h\in \mathfrak{H}%
\right\} $, where $\mathfrak{H}$ is a real separable Hilbert space, with
inner product $\left\langle \cdot ,\cdot \right\rangle $. Hence, we assume
that $X$ describes a family of (uncentered) Gaussian variables, defined on
some probability space $\left( \Omega ,\mathfrak{F,}P\right) $ such that for
all $h_{1},h_{2}\in \mathfrak{H}$, $\left\{ X(h_{1}),X(h_{2})\right\} $ are
jointly Gaussian with mean 
\begin{equation*}
\mathbb{E}X(h)=\left\langle h,f\right\rangle =\int_{\mathbb{S}^{d}}f(x)h(x)dx
\end{equation*}%
and covariance%
\begin{equation*}
\mathbb{E}\left( X(h_{1})-\mathbb{E}X(h_{1})\right) \left( X(h_{2})-\mathbb{E%
}X(h_{2})\right) =\left\langle h_{1},h_{2}\right\rangle \text{ .}
\end{equation*}%
In our case, $\Omega :=\mathbb{S}^{d}$ and $\mathfrak{F}$ is the $\sigma $%
-algebra generated by $X$. We will use $L^{2}\left( \mathbb{S}^{d}\right) $
instead of $L^{2}\left( \mathbb{S}^{d},\mathfrak{F},P\right) $ to simplify
the notation. We shall in fact be concerned with $X_{n}=\left\{ Y\left(
x\right) ,x\in \mathbb{S}^{d}\right\} $, the observed sample path associated
to the process, where we assume that 
\begin{equation*}
\mathbb{E}X_{n}(h)=\left\langle h,f\right\rangle =\int_{\mathbb{S}%
^{d}}f(x)h(x)dx
\end{equation*}%
and covariance%
\begin{equation*}
\mathbb{E}\left( X_{n}(h_{1})-\mathbb{E}X_{n}(h_{1})\right) \left(
X_{n}(h_{2})-\mathbb{E}X_{n}(h_{2})\right) =\frac{1}{n}\left\langle
h_{1},h_{2}\right\rangle \text{ .}
\end{equation*}%
In other words, in order to estimate the unknown function $f$ , on a proper
class of function (in our case, the Besov ball), we will study the estimator
of $f$ which is a function $x\longmapsto \widehat{f}\left( x\right) =%
\widehat{f}\left( x,X_{n}\right) $ defined on the $d$-dimensional sphere and
measurable with respect to the observation $X_{n}$, see again \cite{tsybakov}
and cfr. Remark \ref{whitecorresp}.  

Consider now the usual needlet system $\left\{ \psi _{jk}\right\} _{j,k}$
and let $f\in L^{p}(\mathbb{S}^{d})$; we have the following: 
\begin{eqnarray}
\beta _{jk} &=&\mathbb{E}X_{n}(\psi _{jk})=\left\langle \psi
_{jk},f\right\rangle =\int_{\mathbb{S}^{d}}f(x)\psi _{jk}(x)dx\text{ ,} 
\notag \\
\widehat{\beta }_{jk} &=&X_{n}(\psi _{jk})=\beta _{jk}+\varepsilon _{jk;n}%
\text{ , }  \label{needcoeff2}
\end{eqnarray}%
where%
\begin{eqnarray}
\mathbb{E}\varepsilon _{jk;n} &=&\mathbb{E}\left( X_{n}(\psi _{jk})-\mathbb{E%
}X_{n}(\psi _{jk})\right) =0\text{ , }  \notag \\
E\varepsilon _{jk;n}^{2} &=&\frac{1}{n}\left\langle \psi _{jk},\psi
_{jk}\right\rangle _{L^{2}(\mathbb{S}^{d})}=\frac{1}{n}\left\Vert \psi
_{jk}\right\Vert _{L^{2}(\mathbb{S}^{d})}^{2}\text{ , }  \label{varnoise} \\
E\varepsilon _{jk_{1};n}\varepsilon _{jk_{2};n} &=&\frac{1}{n}\left\langle
\psi _{jk_{1}},\psi _{jk_{2}}\right\rangle _{L^{2}(\mathbb{S}^{d})}  \notag
\\
&=&\frac{1}{n}\frac{\sum_{l}b^{2}(\frac{l}{2^{j}})L_{l}(\left\langle \xi
_{jk_{1}},\xi _{jk_{2}}\right\rangle )}{\sum_{l}b^{2}(\frac{l}{2^{j}}%
)L_{l}(1)}  \notag \\
&=&\frac{1}{n}\frac{\sum_{l}b^{2}(\frac{l}{2^{j}})L_{l}(\left\langle \xi
_{jk_{1}},\xi _{jk_{2}}\right\rangle )}{\sum_{l}b^{2}(\frac{l}{2^{j}})\frac{%
g_{l,d}}{\omega _{d}}}\text{ .}  \notag
\end{eqnarray}

In a formal sense, one could consider the Gaussian white noise measure on
the sphere such that for all $A,B\subset \mathbb{S}^{d},$ we have%
\begin{equation*}
\mathbb{E}W(A)W(B)=\int_{A\cap B}dx\text{ ,}
\end{equation*}%
so that 
\begin{equation*}
\varepsilon _{jk;n}=\frac{1}{n}\int_{\mathbb{S}^{d}}\psi _{jk}(x)W(dx)\text{
,}
\end{equation*}%
as in the Gaussian white noise model on $\left[ 0,1\right] $, described by 
\cite{tsybakov}. 

\begin{remark}
\label{spherecorrespond}Following the same reasoning illustrated in Remark %
\ref{whitecorresp}, asymptotic equivalence holds between our Gaussian white
noise model over $\mathbb{S}^{d}$ and the discrete nonparametric regression
model%
\begin{equation*}
Y_{i}=f\left( X_{i}\right) +\varepsilon _{i}\text{ , }i=1...n\text{ ,}
\end{equation*}%
where in this case $\left\{ X_{i}\right\} _{i=1}^{n}$ are uniform random
locations over $\mathbb{S}^{d}$ and $\left\{ Y_{i}\right\} _{i=1}^{n}$ are
the corresponding observations. By the practical point of view, considering
the Remarks \ref{whitecorresp} and \ref{spherecorrespond}, (cfr. also \cite%
{dgm}), given the dataset $\left\{ X_{i}\right\} _{i=1}^{n}$, the needlets
estimator defined in (\ref{needcoeff2}) corresponds to%
\begin{equation*}
\widehat{\beta }_{jk}:=\frac{1}{n}\sum_{i=1}^{n}Y_{i}\psi _{jk}\left(
X_{i}\right) \text{ }=\frac{1}{n}\sum_{i=1}^{n}\left[ \psi _{jk}\left(
X_{i}\right) f\left( X_{i}\right) +\psi _{jk}\left( X_{i}\right) \varepsilon
_{i}\right] \text{ .}
\end{equation*}%
Furthermore, the unbiasedness is easily verified : 
\begin{eqnarray}
E\left( \widehat{\beta }_{jk}\right)  &=&\frac{1}{n}\sum_{i=1}^{n}E\left[
\psi _{jk}\left( X_{i}\right) f\left( X_{i}\right) +\psi _{jk}\left(
X_{i}\right) \varepsilon _{i}\right]   \notag \\
&=&\int_{\mathbb{S}^{d}}\psi _{jk}\left( x\right) f\left( x\right) dx=\beta
_{jk}\text{ .}  \label{meanbeta}
\end{eqnarray}
\end{remark}

As described above (see also \cite{bkmpAoSb}, \cite{dgm}), $f$ can be
described in terms of needlet coefficients, up to a constant, as 
\begin{equation*}
f=\sum_{j\geq 0}\sum_{k=1}^{N_{j}}\beta _{jk}\psi _{jk}\text{ .}
\end{equation*}

Let us now define the threshold blocks: as anticipated in the Introduction,
differently from \cite{hkp}, the structure itself of the needlet framework
suggests a quite intuitive way to be followed. Let us fix $j>0$: recall that
for each resolution level $j$, we have $N_{j}\approx B^{dj}$ cubature
points. Given the size of the blocks, i.e. the number of cubature points
belonging to each of them - let us say $\ell _{j}$ - we will build using (%
\ref{voronoi}) a set of Voronoi cells, containing $\ell _{j}$ cubature
points. For each cell, we choose a cubature point $\xi _{js}$ to index it:
we define $S_{j}\left( \ell _{j}\right) $ as the number of Voronoi cells
obtained to split cubature points into groups of cardinality $\ell _{j}$.
Let us define the set%
\begin{equation}
R_{j;s}=\left\{ k:\xi _{jk}\in \mathcal{V}\left( \xi _{js}\right) \right\} ,%
\text{ }s=1,...,S_{j}.  \label{blokke}
\end{equation}%
From (\ref{voronoi}), it is immediate to see that each cubature point $\xi
_{jk}$ belongs to a unique Voronoi cell. Obviously, $S_{j}\cdot \ell
_{j}=N_{j}$ .

Let us call, for any integer $p\geq 1$,%
\begin{equation*}
A_{js;p}:=\frac{1}{\ell _{j}}\sum_{k\in R_{j;s}}\beta _{jk}^{p}\text{ ,}
\end{equation*}%
hence we can define the corresponding estimator 
\begin{equation*}
\widehat{A}_{js;p}=\frac{1}{\ell _{j}}\sum_{k\in R_{j;s}}\widehat{\beta }%
_{jk}^{p}\text{ ,}
\end{equation*}%
similar to the ones suggested in \cite{hkp}, Remark 4.7.

We build the following weight function as follows%
\begin{equation*}
w_{js;p}=I\left( \left\vert \widehat{A}_{js;p}\right\vert >\kappa
t_{n}^{p}\right) \text{ ; }
\end{equation*}%
we can define the function estimator as: 
\begin{equation}
\widehat{f}=\sum_{j=0}^{J_{n}}\sum_{s=1}^{S_{j}}\left( \sum_{k\in R_{j;s}}%
\widehat{\beta }_{jk}\psi _{jk}\right) w_{js;p}\text{ ,}  \label{densityest}
\end{equation}%
where:

\begin{itemize}
\item $J_{n}$ is the highest resolution level considered, taken such that 
\begin{subequations}
\begin{equation}
B^{J_{n}}=n^{\frac{1}{d}}\text{ ,}  \label{bbound}
\end{equation}%
consistent with the existent literature (see for instance \cite{WASA}, \cite%
{bkmpAoSb})

\item $\kappa $ is the threshold constant (for more discussions see for
instance \cite{bkmpAoSb}, \cite{dgm}, and \cite{WASA}). As suggested in \cite%
{bkmpAoSb}, $\kappa $ has to be proportional to $M$, the bound of $%
\left\Vert f\right\Vert _{\infty }$, multiplied by a constant $\kappa _{0}$
that can be made explicit with an iterative procedure to count the blocks
not annihilated by the threshold;

\item the scaling factor $t_{n}$, depends on the size of the sample. We will
fix 
\end{subequations}
\begin{equation*}
t_{n}=n^{-\frac{1}{2}}\text{.}
\end{equation*}%
This choice is motivated by two main facts. On one hand, it allows $\widehat{%
f}$ to attain the optimal rate of convergence in the regular zone (cfr.
Theorem \ref{maintheorem} and Remark \ref{remark}). On the other hand, this
choice is consistent with the literature related to thresholding procedures
in needlet frameworks, see \cite{bkmpAoSb} and \cite{dgm}). 

\item The block size will be chosen so that%
\begin{equation*}
\ell _{j}=\left[ N_{j}\right] ^{\eta }\text{ ,}
\end{equation*}%
where $\left[ \cdot \right] $ denotes the integer part and $0<\eta <\frac{1}{%
2}$. The size of the block has to be chosen also considering, on one hand,
the value of the threshold (see above the point related to the choice of $%
\kappa $) and, on the other hand, the number of cubature points at a fixed
resolution level $j$. More details will be given in the Section \ref%
{sec:conclu}.
\end{itemize}

By the practical point of view, given the size of the sample $n$ and the
scale parameter $B$, $J_{n}$ and $t_{n}$ are easily computed. Therefore, the
experimenter should test, for different sizes of the blocks, chosen taking
on account the whole number of cubature points, and for different values of $%
\kappa _{0}$ the number of blocks not annihilated by the procedure. 

\section{Minimax risk rates of convergence\label{sec:minimax}}

This Section aims to describe the performance of the procedure in terms of
the optimality of its convergence rates with respect to general $L^{p}\left( 
\mathbb{S}^{d}\right) $-loss functions: this result is established in the
Theorem \ref{maintheorem}. First of all, we recall the definition of optimal
rate of convergence from \cite{WASA}. We say that an estimator $\widehat{f}$
attains the optimal rate of convergence $R_{n}\left( V,p\right) $ on the
class $V$ for the $L_{p}$-risk if 
\begin{equation*}
\sup_{f\in V}\mathbb{E}\left\Vert \widehat{f}-f\right\Vert
_{L^{p}}^{p}\simeq R_{n}\left( V,p\right) \text{ .}
\end{equation*}%
In our case $V$ corresponds to the Besov ball $\mathcal{B}_{\pi q}^{r}\left(
M\right) $ and $R_{n}\left( V,p\right) $ will assume the form of $n^{-\alpha
\left( r,\pi ,p\right) }$.

\begin{theorem}
\label{maintheorem}Let $f\in \mathcal{B}_{\pi q}^{r}\left( M\right) $, the
Besov ball so that $\left\Vert f\right\Vert _{\mathcal{B}_{\pi q}^{r}\left(
M\right) }\leq M<+\infty $, $r-\frac{d}{\pi }>0$. Consider $\widehat{f}$ as
defined by (\ref{densityest}). For $p\in \mathbb{N}$, there exists a
constant $c_{p}=c_{p}\left( p,r,q,M,B\right) $ such that 
\begin{equation*}
\sup_{f\in \mathcal{B}_{\pi q}^{r}\left( M\right) }\mathbb{E}\left\Vert 
\widehat{f}-f\right\Vert _{L^{p}\left( \mathbb{S}^{d}\right) }^{p}\leq
c_{p}n^{-\alpha \left( r,\pi ,p\right) }\text{ ,}
\end{equation*}%
where 
\begin{equation*}
\alpha \left( r,\pi ,p\right) =\left\{ 
\begin{array}{c}
\frac{rp}{2r+d}\text{ \ \ \ \ \ \ \ \ \ \ \ \ for }\pi \geq \frac{dp}{2r+d}
\\ 
\frac{p\left( r-d\left( \frac{1}{\pi }-\frac{1}{p}\right) \right) }{2\left(
r-d\left( \frac{1}{\pi }-\frac{1}{2}\right) \right) }-\delta \text{ for }\pi
<\frac{dp}{2r+d}%
\end{array}%
\right. \text{ ,}
\end{equation*}%
where $\delta =\delta \left( \eta ,d,\pi ,p,r\right) =\frac{\eta d\left( 1-%
\frac{\pi }{p}\right) }{2\left( r-d\left( \frac{1}{\pi }-\frac{1}{2}\right)
\right) }$.

If $p=+\infty $, there exists a constant $c_{\infty }=c_{\infty }\left(
r,q,M,B\right) $%
\begin{equation*}
\sup_{f\in \mathcal{B}_{\pi q}^{r}\left( M\right) }\mathbb{E}\left\Vert 
\widehat{f}-f\right\Vert _{\infty }\leq c_{\infty }n^{-\alpha \left( r,\pi
,p\right) }\text{ ,}
\end{equation*}%
where 
\begin{equation*}
\alpha \left( r,\pi ,p\right) =\frac{\left( r-\frac{d}{\pi }\right) }{%
2\left( r-d\left( \frac{1}{\pi }-\frac{1}{2}\right) \right) }\text{ .}
\end{equation*}
\end{theorem}

Let us recall that in literature (cfr. for instance \cite{WASA}) the case $%
\pi \geq dp/\left( 2r+d\right) $ is named \emph{regular case}, the other
being referred as \emph{sparse case}.

\begin{remark}
\label{remark}Our results achieve the minimax rates provided in \cite%
{bkmpAoSb} and \cite{dgm}, see also \cite{WASA} just in the regular zone,
while in the sparse zone the rate is worsened by the term $n^{\delta }$.
Note that for convenience our arguments are implemented for integer values $%
p\in \mathbb{N}$. Other real values can be dealt with by interpolation, but
we omit to do it for brevity's sake. Of course, the most relevant case for
practitioners is $p=2,\,$in which case the function certainly belongs to the
regular zone, where our rates are optimal, see also \cite{caibloc2}, \cite%
{donoho2}.
\end{remark}

As in \cite{bkmpAoSb} and \cite{dgm}, the minimax rates are not affected by
the construction over the sphere, which instead is pivotal in the
development of statistical procedures. As mentioned in the Introduction, the
proof of this Theorem makes extensive use of standard techniques (see for
instance \cite{WASA}), modified to exploit the properties of the needlets
described in Section \ref{sec:need}. The procedure is therefore close to the
ones employed in \cite{bkmpAoSb} and \cite{dgm}, the main differences
concerning the probabilistic inequalities in Section \ref{sec:aux} and some
of the pivotal steps of the proof remarked in Section \ref{sec:proof}.

\section{Auxiliary Results\label{sec:aux}}

This Section collects the probabilistic inequalities necessary to prove
Theorem \ref{maintheorem}.

\begin{lemma}
\label{mainlemma} Consider $\widehat{\beta }_{jk}$ as described in \ref%
{needcoeff2}. There exist constants $C_{p},C_{\infty },C_{A}$ such that, for 
$B^{j}\leq n^{\frac{1}{2}}$, $j=0,...,J_{n}$,%
\begin{equation}
\mathbb{E}\left[ \left\vert \widehat{\beta }_{jk}-\beta _{jk}\right\vert ^{p}%
\right] \leq C_{p}n^{-p/2}\text{ , }p\geq 1  \label{E_p}
\end{equation}%
\begin{equation}
\mathbb{E}\left[ \sup_{k=1,...,N_{j}}\left\vert \widehat{\beta }_{jk}-\beta
_{jk}\right\vert ^{p}\right] \leq C_{\infty }\left( j+1\right) ^{p}n^{-p/2}%
\text{, }p\geq 1\text{ ,}  \label{E_inf}
\end{equation}%
and for all $\gamma >0,$ $p\in \mathbb{N}$, there exists $\kappa >0$ such
that%
\begin{equation}
\mathbb{P}\left( \left\vert \widehat{A}_{js;p}-A_{js;p}\right\vert >\kappa
t_{n}^{p}\right) \leq C_{p,\gamma }\frac{1}{n^{\gamma }}\text{ .}
\label{P_A}
\end{equation}
\end{lemma}

\begin{proof}
First of all, consider that, from Equations (\ref{needcoeff2}) and (\ref%
{varnoise}), we have 
\begin{eqnarray*}
\mathbb{E}\left( \left\vert \widehat{\beta }_{jk}-\beta _{jk}\right\vert
^{p}\right) &=&\mathbb{E}\left( \left\vert \varepsilon _{jk;n}\right\vert
^{p}\right) \\
&=&\left( Var\left( \varepsilon _{jk;n}\right) \right) ^{\frac{p}{2}}\frac{%
2^{\frac{p}{2}}\Gamma \left( \frac{p+1}{2}\right) }{\sqrt{\pi }} \\
\text{ } &=&\frac{1}{n^{\frac{p}{2}}}\left\Vert \psi _{jk}\right\Vert
_{L^{2}(\mathbb{S}^{d})}^{p}\frac{2^{\frac{p}{2}}\Gamma \left( \frac{p+1}{2}%
\right) }{\sqrt{\pi }} \\
&=&O\left( n^{-\frac{p}{2}}\right) \text{ ,}\ 
\end{eqnarray*}%
to obtain (\ref{E_p}). Now, for Mill's inequality, if $Z\sim N\left(
0,1\right) $, we have $\mathbb{P}\left( \left\vert Z\right\vert \geq
x\right) \leq \sqrt{2/\pi }\exp \left( -x^{2}/2\right) /x\,$. Hence, from (%
\ref{varnoise}), we obtain 
\begin{eqnarray*}
\mathbb{P}\left( \left\vert \varepsilon _{jk;n}\right\vert \geq x\right) &=&2%
\mathbb{P}\left( \left\vert Z\right\vert \geq \sqrt{n}x\right) \\
&\leq &\sqrt{\frac{2}{\pi }}\frac{e^{-\frac{nx^{2}}{2}}}{\sqrt{n}x} \\
&\leq &C_{\varepsilon }e^{-\frac{nx^{2}}{2}}\text{ .}
\end{eqnarray*}%
On the other hand, we have%
\begin{eqnarray*}
\mathbb{E}\left[ \sup_{k=1,...,N_{j}}\left\vert \widehat{\beta }_{jk}-\beta
_{jk}\right\vert ^{p}\right] &=&\int_{\mathbb{R}^{+}}x^{p-1}\mathbb{P}\left(
\sup_{k=1,...,N_{j}}\left\vert \widehat{\beta }_{jk}-\beta _{jk}\right\vert
\geq x\right) dx \\
&=&\int_{\mathbb{R}^{+}}x^{p-1}\mathbb{P}\left(
\sup_{k=1,...,N_{j}}\left\vert \varepsilon _{jk;n}\right\vert \geq x\right)
dx \\
&=&E_{1}+E_{2}\text{ ,}
\end{eqnarray*}%
where%
\begin{eqnarray*}
E_{1} &=&\int_{0\leq x\leq \frac{2\sqrt{2}}{\sqrt{n}}j}x^{p-1}dx\text{ ,} \\
E_{2} &=&C\int_{x>\frac{2\sqrt{2}}{\sqrt{n}}j}x^{p-1}B^{2j}\max_{k}\mathbb{P}%
\left( \left\vert \varepsilon _{jk;n}\right\vert \geq x\right) dx\text{ .}
\end{eqnarray*}%
We can easily see that 
\begin{equation*}
E_{1}=C_{1}j^{p}n^{-\frac{p}{2}}\text{ ,}
\end{equation*}%
while on the other hand, considering that for $x>2\sqrt{2/n}j$ 
\begin{equation*}
B^{2j}e^{-\frac{nx^{2}}{2}}\leq e^{-\frac{nx^{2}}{4}-\frac{nx^{2}}{4}%
+2j}\leq e^{-\frac{nx^{2}}{4}}\text{,}
\end{equation*}%
we obtain 
\begin{eqnarray*}
E_{2} &\leq &C\int_{x>\frac{2\sqrt{2}}{\sqrt{n}}j}x^{p-1}B^{2j}e^{-\frac{%
nx^{2}}{2}}dx \\
&\leq &C_{2}n^{-\frac{p}{2}}\text{,}
\end{eqnarray*}%
so we achieve (\ref{E_inf}). In order to prove (\ref{P_A}), we write%
\begin{equation*}
\mathbb{P}\left( \left\vert \widehat{A}_{js;p}-A_{js;p}\right\vert >\kappa
t_{n}^{p}\right)
\end{equation*}%
\begin{equation*}
=\mathbb{P}\left\{ \left( \frac{1}{\ell _{j}}\sum_{k=1}^{\ell _{j}}\left( 
\widehat{\beta }_{jk}^{p}-\mathbb{E}\widehat{\beta }_{jk}^{p}\right) \right)
^{1/p}>\frac{\kappa }{\sqrt{n}}\right\} \text{ .}
\end{equation*}%
Define%
\begin{equation*}
\widetilde{\beta }_{jk}:=\sqrt{n}\beta _{jk}+\sqrt{n}\varepsilon _{jk;n}=%
\sqrt{n}\beta _{jk}+\varepsilon _{jk}\text{ ,}
\end{equation*}%
where 
\begin{equation*}
\varepsilon _{jk}:=\sqrt{n}\varepsilon _{jk;n}\text{ ;}
\end{equation*}%
our aim is hence to study the behaviour of the terms of the form 
\begin{equation}
\left( \frac{1}{\ell _{j}}\sum_{k=1}^{\ell _{j}}\varepsilon _{jk}^{p}+\frac{p%
\sqrt{n}}{\ell _{j}}\sum_{k=1}^{\ell _{j}}\beta _{jk}\varepsilon
_{jk}^{p-1}+...+\frac{pn^{(p-1)/2}}{\ell _{j}}\sum_{k=1}^{\ell _{j}}\beta
_{jk}^{p-1}\varepsilon _{jk}\right) ^{1/p}\text{.}  \label{pino}
\end{equation}%
Observe that:%
\begin{equation*}
\frac{1}{\ell _{j}}\sum_{k=1}^{\ell _{j}}\beta _{jk}^{p-1}\varepsilon
_{jk}\leq \left( \frac{1}{\ell _{j}}\sum_{k=1}^{\ell _{j}}\beta
_{jk}^{2p-2}\right) ^{1/2}\left( \frac{1}{\ell _{j}}\sum_{k=1}^{\ell
_{j}}\varepsilon _{jk}^{2}\right) ^{1/2}\text{ ;}
\end{equation*}%
we have that%
\begin{equation*}
\sum_{k=1}^{\ell _{j}}\beta _{jk}^{2p-2}\leq \sum_{k=1}^{N_{j}}\beta
_{jk}^{2p-2}=O\left( B^{-js}B^{-j\frac{d}{2}(1-\frac{1}{p-1})}\right)
=O\left( B^{-js}B^{-j\frac{d}{2}(\frac{p-2}{p-1})}\right) \text{ .}
\end{equation*}%
\begin{equation*}
O\left( B^{-js}B^{-j\frac{d}{2}(\frac{p-2}{p-1})}\right)
\end{equation*}%
On the other hand, by Lemma \ref{epsilonlemma}, for all $p,\gamma >0$, there
exists $\kappa >0$ such that 
\begin{equation*}
\mathbb{P}\left\{ \frac{1}{\ell _{j}}\sum_{k=1}^{\ell _{j}}\left\vert
\varepsilon _{jk}\right\vert ^{p}>\kappa \right\} \leq \frac{C_{p,\gamma }}{%
\ell _{j}^{\gamma /2}}\text{ .}
\end{equation*}%
Hence, we obtain 
\begin{equation*}
\frac{pn^{(p-1)/2}}{\ell _{j}}\sum_{k=1}^{\ell _{j}}\beta
_{jk}^{p-1}\varepsilon _{jk}\leq C\frac{n^{(p-1)/2}}{\ell _{j}^{\frac{\gamma
+2}{2}}}B^{-j(\frac{d}{2}\frac{p-2}{p-1}+s)}\text{ .}
\end{equation*}%
By choosing suitable $s$ and $\gamma $, we have 
\begin{equation*}
\frac{pn^{(p-1)/2}}{\ell _{j}}\sum_{k=1}^{\ell _{j}}\beta
_{jk}^{p-1}\varepsilon _{jk}=o\left( \frac{1}{\ell _{j}}\sum_{k=1}^{\ell
_{j}}\varepsilon _{jk}^{p}\right) \text{ .}
\end{equation*}%
The same holds for all the other mixed terms in Equation (\ref{pino}).
\end{proof}

\begin{lemma}
\label{epsilonlemma}Assume that $\mathbb{E}\varepsilon _{jk}=0,$ $\mathbb{E}%
\varepsilon _{jk}^{2}=1,$ and%
\begin{equation*}
\mathbb{E}\varepsilon _{jk_{1}}\varepsilon _{jk_{2}}\leq \frac{C_{M}}{%
\left\{ 1+B^{\frac{d}{2}j}d(\xi _{jk_{1}},\xi _{jk_{2}})\right\} ^{M}}\text{
, for all }M>0\text{ .}
\end{equation*}%
For all $p\in \mathbb{N},$ $\gamma >0$ there exists $\kappa >0$ such that%
\begin{equation*}
\mathbb{P}\left\{ \frac{1}{\ell _{j}}\sum_{k=1}^{\ell _{j}}\left\vert
\varepsilon _{jk}\right\vert ^{p}>\kappa \right\} \leq \frac{C_{p,\gamma }}{%
\ell _{j}^{\gamma /2}}\text{ .}
\end{equation*}
\end{lemma}

\begin{proof}
Without loss of generality we can take $p$ to be even; note indeed that%
\begin{equation*}
\mathbb{P}\left\{ \frac{1}{\ell _{j}}\sum_{k=1}^{\ell _{j}}\left\vert
\varepsilon _{jk}\right\vert ^{p}>\kappa \right\} \leq \mathbb{P}\left\{ 
\frac{1}{\ell _{j}}\sum_{k=1}^{\ell _{j}}\varepsilon _{jk}^{2p}>\kappa
^{2}\right\} \text{ .}
\end{equation*}%
Let us rewrite%
\begin{equation*}
\frac{1}{\ell _{j}}\sum_{k=1}^{\ell _{j}}\varepsilon _{jk}^{p}=\mathbb{E}%
\varepsilon _{jk}^{p}+\sum_{\tau =1}^{p}c_{\tau }H_{\tau }(\varepsilon _{jk})%
\text{ , }
\end{equation*}%
whence%
\begin{equation*}
\mathbb{P}\left\{ \frac{1}{\ell _{j}}\sum_{k=1}^{\ell _{j}}\varepsilon
_{jk}^{p}>p(\kappa +E\varepsilon _{jk}^{p})\right\} \leq \sum_{\tau =1}^{p}%
\mathbb{P}\left\{ \frac{1}{\ell _{j}}\sum_{k=1}^{\ell _{j}}H_{\tau
}(\varepsilon _{jk})>\frac{\kappa }{c_{\tau }}\right\} \text{ .}
\end{equation*}%
By the Markov's inequality, the result will hence follow if we prove that%
\begin{equation*}
\mathbb{E}\left[ \frac{1}{\ell _{j}}\sum_{k=1}^{\ell _{j}}H_{\tau
}(\varepsilon _{jk})\right] ^{\gamma }\leq \frac{C}{\ell _{j}^{\gamma /2}}%
\text{ .}
\end{equation*}%
Now let us take for notational simplicity $\tau =2;$ the argument for the
other terms is identical. We have%
\begin{equation*}
\mathbb{E}\left[ \frac{1}{\ell _{j}}\sum_{k=1}^{\ell _{j}}H_{\tau
}(\varepsilon _{jk})\right] ^{\gamma }=\frac{1}{\ell _{j}^{\gamma }}%
\sum_{k_{1},..,k_{\gamma }=1}^{\ell _{j}}\mathbb{E}\left\{ H_{\tau
}(\varepsilon _{jk_{1}})...H_{\tau }(\varepsilon _{jk_{\gamma }})\right\}
\end{equation*}%
\begin{equation*}
=\frac{1}{\ell _{j}^{\gamma }}\left\{ \sum_{k_{1}k_{2}}^{\ell _{j}}\left[ 
\mathbb{E}(\varepsilon _{jk_{1}}\varepsilon _{jk_{2}})\right] ^{2}\right\}
^{\gamma /2}
\end{equation*}%
\begin{equation*}
+\frac{1}{\ell _{j}^{\gamma }}\left\{ \sum_{k_{1}k_{2}}^{\ell _{j}}\left[ 
\mathbb{E}(\varepsilon _{jk_{1}}\varepsilon _{jk_{2}})\right] ^{2}\right\} ^{%
\frac{\gamma }{2}-2}\left\{ \sum_{k_{1}...k_{4}}^{\ell _{j}}\mathbb{E}%
(\varepsilon _{jk_{1}}\varepsilon _{jk_{2}})\mathbb{E}(\varepsilon
_{jk_{2}}\varepsilon _{jk_{3}})\mathbb{E}(\varepsilon _{jk_{3}}\varepsilon
_{jk_{4}})\mathbb{E}(\varepsilon _{jk_{4}}\varepsilon _{jk_{1}})\right\}
\end{equation*}%
\begin{equation*}
+\frac{1}{\ell _{j}^{\gamma }}\left\{ \sum_{k_{1}k_{2}}^{\ell _{j}}\left[ 
\mathbb{E}(\varepsilon _{jk_{1}}\varepsilon _{jk_{2}})\right] ^{2}\right\} ^{%
\frac{\gamma }{2}-4}\left\{ \sum_{k_{1}...k_{6}}^{\ell _{j}}\mathbb{E}%
(\varepsilon _{jk_{1}}\varepsilon _{jk_{2}})...\mathbb{E}(\varepsilon
_{jk_{6}}\varepsilon _{jk_{1}})\right\}
\end{equation*}%
\begin{equation*}
+\frac{1}{\ell _{j}^{\gamma }}\left\{ \sum_{k_{1}...k_{\gamma }}^{\ell _{j}}%
\mathbb{E}(\varepsilon _{jk_{1}}\varepsilon _{jk_{2}})...\mathbb{E}%
(\varepsilon _{jk_{\gamma }}\varepsilon _{jk_{1}})\right\}
\end{equation*}%
\begin{equation*}
=O(\ell _{j}^{-\gamma /2})+O(\ell _{j}^{-\frac{\gamma }{2}-1})+...+O(\ell
_{j}^{-\gamma +1})\text{ ,}
\end{equation*}%
because%
\begin{eqnarray*}
\sum_{k_{1}...k_{\gamma }}^{\ell _{j}}E(\varepsilon _{jk_{1}}\varepsilon
_{jk_{2}})...E(\varepsilon _{jk_{\gamma }}\varepsilon _{jk_{1}}) &\leq
&\sum_{k_{1}...k_{q}}^{\ell _{j}}\left\vert E(\varepsilon
_{jk_{1}}\varepsilon _{jk_{2}})\right\vert ...\left\vert E(\varepsilon
_{jk_{\gamma -1}}\varepsilon _{jk_{\gamma }})\right\vert \\
&\leq &\ell _{j}\left\{ \sum_{k_{2}}^{\ell _{j}}\left\vert E(\varepsilon
_{jk_{1}}\varepsilon _{jk_{2}})\right\vert \right\} ^{\gamma -1}=O(\ell _{j})%
\text{ .}
\end{eqnarray*}
\end{proof}

\section{Proof of Theorem \protect\ref{maintheorem} (upper bound) \label%
{sec:proof}}

This Section will describe in details the proof of the Theorem \ref%
{maintheorem}. As previously mentioned, some of the passages of this proof
will be very close to those developed for local thresholding described in 
\cite{bkmpAoSb} and \cite{dgm}, hence we will omit them. First of all,
observe that%
\begin{equation*}
\sum_{s=1}^{S_{j}}\sum_{k\in R_{j;s}}=\sum_{k=1}^{N_{j}}\text{ .}
\end{equation*}%
Standard calculations (see for instance \cite{WASA}) lead to: 
\begin{eqnarray*}
\mathbb{E}\left\Vert \widehat{f}-f\right\Vert _{L^{p}\left( \mathbb{S}%
^{d}\right) }^{p} &=&\mathbb{E}\left\Vert
\sum_{j=0}^{J_{n}}\sum_{s=1}^{S_{j}}\left( \sum_{k\in R_{j;s}}\widehat{\beta 
}_{jk}\psi _{jk}\right) w_{js;p}-\sum_{j\geq 0}\sum_{k=1}^{N_{j}}\beta
_{jk}\psi _{jk}\right\Vert _{L^{p}\left( \mathbb{S}^{d}\right) }^{p} \\
&=&\mathbb{E}\left\Vert \sum_{j=0}^{J_{n}}\sum_{s=1}^{S_{j}}\sum_{k\in
R_{j;s}}\left( w_{js;p}\widehat{\beta }_{jk}-\beta _{jk}\right) \psi
_{jk}-\sum_{j>J_{n}}\sum_{k=1}^{N_{j}}\beta _{jk}\psi _{jk}\right\Vert
_{L^{p}\left( \mathbb{S}^{d}\right) }^{p} \\
&\leq &2^{p-1}\left( \mathbb{E}\left\Vert
\sum_{j=0}^{J_{n}}\sum_{s=1}^{S_{j}}\sum_{k\in R_{j;s}}\left( w_{js;p}%
\widehat{\beta }_{jk}-\beta _{jk}\right) \psi _{jk}\right\Vert _{L^{p}\left( 
\mathbb{S}^{d}\right) }^{p}\right. \\
&&+\left. \left\Vert \sum_{j>J_{n}}\sum_{k=1}^{N_{j}}\beta _{jk}\psi
_{jk}\right\Vert _{L^{p}\left( \mathbb{S}^{d}\right) }^{p}\right) \\
&=&:I+II\text{ .}
\end{eqnarray*}%
Consider now the two different cases mentioned in Section \ref{sec:minimax}.

\textbf{CASE\ I:} \emph{Regular\ Case}

Consider $p<+\infty $. For $p\leq \pi $, we have $\mathcal{B}_{\pi
q}^{r}\subset \mathcal{B}_{pq}^{r}$: we therefore take $\pi =p\,$. Consider
instead the case $p>\pi $: we use the embedding $\mathcal{B}_{\pi
q}^{r}\subset \mathcal{B}_{pq}^{r-d\left( \frac{1}{p}-\frac{1}{\pi }\right)
} $, and moreover we assume%
\begin{equation*}
r\geq \frac{d}{p},\frac{r}{2r+d}=\frac{rp}{\left( 2r+d\right) p}\leq \frac{%
r\pi }{dp}\text{ ,}
\end{equation*}%
we have as in \cite{bkmpAoSb}, \cite{dgm}, that%
\begin{equation*}
II\leq O\left( n^{-\frac{pr}{2r+d}}\right) \text{ ,}
\end{equation*}%
as claimed.

About the variance term, from the Lo\`{e}ve's inequality we have%
\begin{eqnarray*}
I &\leq &C\mathbb{E}\left\Vert
\sum_{j=0}^{J_{n}}\sum_{s=1}^{S_{j}}\sum_{k\in R_{j;s}}\left( w_{js;p}%
\widehat{\beta }_{jk}-\beta _{jk}\right) \psi _{jk}\right\Vert _{L^{p}\left( 
\mathbb{S}^{d}\right) }^{p} \\
&\leq &CJ_{n}^{p-1}\sum_{j\leq J_{n}}\mathbb{E}\left\Vert
\sum_{s=1}^{S_{j}}\sum_{k\in R_{j;s}}\left( w_{js;p}\widehat{\beta }%
_{jk}-\beta _{jk}\right) \psi _{jk}\right\Vert _{L^{p}\left( \mathbb{S}%
^{d}\right) }^{p}\text{ .}
\end{eqnarray*}%
As described in \cite{bkmpAoSb}, see also \cite{marpecbook}, we have the
following needlet property:%
\begin{equation*}
\mathbb{E}\left\Vert \sum_{k}\alpha _{k}\psi _{jk}\right\Vert _{L^{p}\left( 
\mathbb{S}^{d}\right) }^{p}=\left\Vert \psi _{jk}\right\Vert _{L^{p}\left( 
\mathbb{S}^{d}\right) }^{p}\sum_{k}\mathbb{E}\left\Vert \alpha
_{k}\right\Vert _{L^{p}\left( \mathbb{S}^{d}\right) }^{p}\text{ .}
\end{equation*}%
Hence, we obtain 
\begin{equation}
\sum_{j\leq J_{n}}\mathbb{E}\left\Vert \sum_{s=1}^{S_{j}}\sum_{k\in
R_{j;s}}\left( w_{js;p}\widehat{\beta }_{jk}-\beta _{jk}\right) \psi
_{jk}\right\Vert _{L^{p}\left( \mathbb{S}^{d}\right) }^{p}  \label{pipi}
\end{equation}%
\begin{eqnarray*}
&=&\sum_{j\leq J_{n}}\mathbb{E}\left\Vert \sum_{s=1}^{S_{j}}\sum_{k\in
R_{j;s}}\left( w_{js;p}\widehat{\beta }_{jk}-\beta _{jk}\right) \psi
_{jk}I\left( \left\vert \widehat{A}_{js;p}\right\vert \geq t_{n}^{p}\right)
I\left( \left\vert A_{js;p}\right\vert \geq \frac{t_{n}^{p}}{2}\right)
\right\Vert _{L^{p}\left( \mathbb{S}^{d}\right) }^{p} \\
&&+\sum_{j\leq J_{n}}\mathbb{E}\left\Vert \sum_{s=1}^{S_{j}}\sum_{k\in
R_{j;s}}\left( w_{js;p}\widehat{\beta }_{jk}-\beta _{jk}\right) \psi
_{jk}I\left( \left\vert \widehat{A}_{js;p}\right\vert \geq t_{n}^{p}\right)
I\left( \left\vert A_{js;p}\right\vert <\frac{t_{n}^{p}}{2}\right)
\right\Vert _{L^{p}\left( \mathbb{S}^{d}\right) }^{p} \\
&&+\sum_{j\leq J_{n}}\mathbb{E}\left\Vert \sum_{s=1}^{S_{j}}\sum_{k\in
R_{j;s}}\left( w_{js;p}\widehat{\beta }_{jk}-\beta _{jk}\right) \psi
_{jk}I\left( \left\vert \widehat{A}_{js;p}\right\vert <t_{n}^{p}\right)
I\left( \left\vert A_{js;p}\right\vert \geq 2t_{n}^{p}\right) \right\Vert
_{L^{p}\left( \mathbb{S}^{d}\right) }^{p} \\
&&+\sum_{j\leq J_{n}}\mathbb{E}\left\Vert \sum_{s=1}^{S_{j}}\sum_{k\in
R_{j;s}}\left( w_{js;p}\widehat{\beta }_{jk}-\beta _{jk}\right) \psi
_{jk}I\left( \left\vert \widehat{A}_{js;p}\right\vert <t_{n}^{p}\right)
I\left( \left\vert A_{js;p}\right\vert <2t_{n}^{p}\right) \right\Vert
_{L^{p}\left( \mathbb{S}^{d}\right) }^{p}
\end{eqnarray*}%
\begin{eqnarray*}
&\leq &C\left\{ \sum_{j\leq J_{n}}\sum_{s=1}^{S_{j}}\sum_{k\in
R_{j;s}}\left\Vert \psi _{jk}\right\Vert _{L^{p}\left( \mathbb{S}^{d}\right)
}^{p}\right. \mathbb{E}\left[ \left( \widehat{\beta }_{jk}-\beta
_{jk}\right) ^{p}I\left( \left\vert \widehat{A}_{js;p}\right\vert \geq
t_{n}^{p}\right) I\left( \left\vert A_{js;p}\right\vert \geq \frac{t_{n}^{p}%
}{2}\right) \right] \\
&&+\sum_{j\leq J_{n}}\sum_{s=1}^{S_{j}}\sum_{k\in R_{j;s}}\left\Vert \psi
_{jk}\right\Vert _{L^{p}\left( \mathbb{S}^{d}\right) }^{p}\mathbb{E}\left[
\left( \widehat{\beta }_{jk}-\beta _{jk}\right) ^{p}I\left( \left\vert 
\widehat{A}_{js;p}\right\vert \geq t_{n}^{p}\right) I\left( \left\vert
A_{js;p}\right\vert <\frac{t_{n}^{p}}{2}\right) \right] \\
&&+\sum_{j\leq J_{n}}\sum_{s=1}^{S_{j}}\sum_{k\in R_{j;s}}\left\Vert \psi
_{jk}\right\Vert _{L^{p}\left( \mathbb{S}^{d}\right) }^{p}\left\vert \beta
_{jk}\right\vert ^{p}\mathbb{E}\left[ I\left( \left\vert \widehat{A}%
_{js;p}\right\vert <t_{n}^{p}\right) I\left( \left\vert A_{js;p}\right\vert
\geq 2t_{n}^{p}\right) \right] \\
&&+\left. \sum_{j\leq J_{n}}\sum_{s=1}^{S_{j}}\sum_{k\in R_{j;s}}\left\Vert
\psi _{jk}\right\Vert _{L^{p}\left( \mathbb{S}^{d}\right) }^{p}\left\vert
\beta _{jk}\right\vert ^{p}\mathbb{E}\left[ I\left( \left\vert \widehat{A}%
_{js;p}\right\vert <t_{n}^{p}\right) I\left( \left\vert A_{js;p}\right\vert
<2t_{n}^{p}\right) \right] \right\}
\end{eqnarray*}%
\begin{equation*}
=Aa+Au+Ua+Uu\text{ .}
\end{equation*}

The procedure follows these guidelines: we have to split (\ref{pipi}) into
four terms: in one of them, $Aa$, both the $\widehat{A}_{js;p}$ and $%
A_{js;p} $ are supposed to be bigger than the respective threshold; in
another one, $Uu$, they are both smaller and in the last two of them, $Au$
and $Ua$, the distance between $\widehat{A}_{js;p}$ and $A_{js;p}$ is shown
to be bigger than a suitable threshold. In the first two cases, in order to
achieve the minimax rate of convergence, we will split these terms into two
parts and we will show the convergence of each part by using mainly (\ref%
{boundnorm}), (\ref{E_p}) and (\ref{E_inf}). The convergence of the last two
terms will be instead proved by applying (\ref{P_A}).

Observe that%
\begin{eqnarray*}
Aa &\leq &C\sum_{j\leq J_{n}}\sum_{s=1}^{S_{j}}\sum_{k\in R_{j;s}}\left\Vert
\psi _{jk}\right\Vert _{L^{p}\left( \mathbb{S}^{d}\right) }^{p}\mathbb{E}%
\left[ \left\vert \widehat{\beta }_{jk}-\beta _{jk}\right\vert ^{p}\right]
I\left( \left\vert A_{js;p}\right\vert \geq \frac{t_{n}^{p}}{2}\right) \\
&\leq &C\sum_{j\leq J_{n}}\sum_{s=1}^{S_{j}}\sum_{k\in R_{j;s}}B^{dj\left( 
\frac{p}{2}-1\right) }I\left( \left\vert A_{js;p}\right\vert \geq \frac{%
t_{n}^{p}}{2}\right) \mathbb{E}\left[ \left\vert \widehat{\beta }_{jk}-\beta
_{jk}\right\vert ^{p}\right] \text{ .}
\end{eqnarray*}%
As in \cite{bkmpAoSb}, \cite{dgm}, we fix $J_{1n}$ such that 
\begin{equation*}
B^{J_{1n}}=O\left( n^{\frac{1}{2r+d}}\right) \text{ ;}
\end{equation*}%
simple calculations show that%
\begin{equation*}
\sum_{j=J_{1n}}^{J_{n}}\ell _{j}B^{dj\left( \frac{p}{2}-1\right)
}\sum_{s=1}^{S_{j}}I\left( \left\vert A_{js;p}\right\vert \geq \frac{%
t_{n}^{p}}{2}\right)
\end{equation*}%
\begin{eqnarray*}
&\leq &\sum_{j=J_{1n}}^{J_{n}}\ell _{j}B^{dj\left( \frac{p}{2}-1\right)
}\sum_{s=1}^{S_{j}}\left\vert A_{js;p}\right\vert \left( \frac{t_{n}^{p}}{2}%
\right) ^{-1} \\
&\leq &Ct_{n}^{-p}\sum_{j=J_{1n}}^{J_{n}}B^{dj\left( \frac{p}{2}-1\right)
}\sum_{k=1}^{N_{j}}\left\vert \beta _{jk}\right\vert ^{p} \\
&\leq &Cn^{\frac{p}{2}}\sum_{j=J_{1n}}^{J_{n}}\sum_{k_{1}=1}^{N_{j}}\left%
\vert \beta _{jk_{1}}\right\vert ^{p}B^{dj\left( \frac{p}{2}-1\right) }\text{
.}
\end{eqnarray*}%
Because $f\in \mathcal{B}_{pq}^{r}$, we have%
\begin{equation*}
\sum_{k=1}^{N_{j}}\left\vert \beta _{jk_{1}}\right\vert ^{p}B^{dj\left( 
\frac{p}{2}-1\right) }=C\sum_{k=1}^{N_{j}}\left\vert \beta
_{jk_{1}}\right\vert ^{p}\left\Vert \psi _{jk}\right\Vert _{p}^{p}\leq
CB^{-prj}\text{ ,}
\end{equation*}%
and, as in \cite{bkmpAoSb}, \cite{dgm} 
\begin{equation*}
n^{\frac{p}{2}}\sum_{j=J_{1n}}^{J_{n}}\sum_{k_{1}=1}^{N_{j}}\left\vert \beta
_{jk_{1}}\right\vert ^{p}B^{dj\left( \frac{p}{2}-1\right) }\leq Cn^{\frac{p}{%
2r+d}}\leq B^{pJ_{1n}}\text{ .}
\end{equation*}%
so that%
\begin{equation*}
\sum_{j=J_{1n}}^{J_{n}}\ell _{j}B^{dj\left( \frac{p}{2}-1\right)
}\sum_{s=1}^{S_{j}}I\left( \left\vert A_{js;p}\right\vert \geq \frac{%
t_{n}^{p}}{2}\right) \leq B^{pJ_{1n}}\text{.}
\end{equation*}

Hence, we obtain 
\begin{eqnarray*}
Aa &\leq &Cn^{-p/2}\left( \sum_{j\leq J_{1n}}\sum_{s=1}^{S_{j}}\ell
_{j}B^{dj\left( \frac{p}{2}-1\right) }I\left( \left\vert A_{js;p}\right\vert
\geq \frac{t_{n}^{p}}{2}\right) \right. \\
&&+\left. \sum_{j=J_{1n}}^{J_{n}}\sum_{s=1}^{S_{j}}\ell _{j}B^{dj\left( 
\frac{p}{2}-1\right) }I\left( \left\vert A_{js;p}\right\vert \geq \frac{%
t_{n}^{p}}{2}\right) \right) \\
&\leq &Cn^{-p/2}\left( \sum_{j\leq J_{1n}}B^{j\frac{d}{2}p}+B^{pJ_{1n}}%
\right) \\
&\leq &Cn^{-p/2}B^{pJ_{1n}}=Cn^{\frac{-pr}{2r+d}}\text{ .}
\end{eqnarray*}

Consider now the term $Uu$. We have that%
\begin{eqnarray*}
Uu &\leq &C\sum_{j\leq J_{n}}\sum_{s=1}^{S_{j}}\sum_{k\in R_{j;s}}\left\Vert
\psi _{jk}\right\Vert _{L^{p}\left( \mathbb{S}^{d}\right) }^{p}\left\vert
\beta _{jk}\right\vert ^{p}I\left( \left\vert A_{js;p}\right\vert
<2t_{n}^{p}\right) \\
&\leq &C\sum_{j\leq J_{n}}l_{j}B^{dj\left( \frac{p}{2}-1\right)
}\sum_{s=1}^{S_{j}}A_{js;p}I\left( \left\vert A_{js;p}\right\vert
<2t_{n}^{p}\right) \\
&\leq &C\left[ \sum_{j\leq J_{1n}}N_{j}B^{dj\left( \frac{p}{2}-1\right)
}2t_{n}^{p}+\sum_{j=J_{1n}}^{J_{n}}\sum_{k=1}^{N_{j}}\left\vert \beta
_{jk}\right\vert ^{p}\left\Vert \psi _{jk}\right\Vert _{L^{p}\left( \mathbb{S%
}^{d}\right) }^{p}\right] \\
&\leq &C\left[ n^{-\frac{p}{2}}B^{pJ_{1n}}+B^{-prJ_{1n}}\right] =O\left( n^{-%
\frac{pr}{2\left( r+1\right) }}\right) \text{ .}
\end{eqnarray*}

Let us study now $Au$ and $Ua$. As in \cite{bkmpAoSb}, \cite{dgm}, we have 
\begin{eqnarray}
Au &\leq &\sum_{j\leq J_{n}}\sum_{s=1}^{S_{j}}\sum_{k\in R_{j;s}}\left\Vert
\psi _{jk}\right\Vert _{L^{p}\left( \mathbb{S}^{d}\right) }^{p}\left( 
\mathbb{E}\left[ \left\vert \widehat{\beta }_{jk}-\beta _{jk}\right\vert
^{2p}\right] \right) ^{\frac{1}{2}}  \notag \\
&&\times \left( \mathbb{P}\left( \left\vert \widehat{A}_{js;p}-A_{js;p}%
\right\vert \geq \frac{\kappa n^{-\frac{p}{2}}}{2}\right) \right) ^{\frac{1}{%
2}}  \notag \\
&\leq &CB^{pJ_{n}}n^{-\frac{p}{2}}n^{-\gamma }\leq Cn^{-\gamma }\text{ ;} 
\notag \\
Ua &\leq &\sum_{j\leq J_{n}}\sum_{s=1}^{S_{j}}\sum_{k\in R_{j;s}}\left\Vert
\psi _{jk}\right\Vert _{L^{p}\left( \mathbb{S}^{d}\right) }^{p}\left\vert
\beta _{jk}\right\vert ^{p}\left( \mathbb{P}\left( \left\vert \widehat{A}%
_{js;p}-A_{js;p}\right\vert \geq \kappa n^{-\frac{p}{2}}\right) \right) 
\notag \\
&\leq &Cn^{-\gamma }\left\Vert F\right\Vert _{p}^{p}\text{ .}  \label{pivo}
\end{eqnarray}%
Because for $r\geq 1$, we have 
\begin{equation*}
n^{-\gamma }\leq n^{-\frac{1}{2}}\leq n^{\frac{-r}{2r+d}}\text{ ,}
\end{equation*}%
the result is proved.

Consider now \thinspace $p=+\infty $: we assume now $f\in \mathcal{B}%
_{\infty ,\infty }^{r}$, to obtain%
\begin{eqnarray*}
\mathbb{E}\left\Vert \widehat{f}-f\right\Vert _{\infty } &\leq &\mathbb{E}%
\left\Vert \sum_{j=0}^{J_{n}}\sum_{k=1}^{N_{j}}\left( w_{j;p}\widehat{\beta }%
_{jk}-\beta _{jk}\right) \psi _{jk}\right\Vert _{L^{\infty }\left( \mathbb{S}%
^{d}\right) }+\left\Vert \sum_{j>J_{n}}\sum_{k=1}^{N_{j}}\beta _{jk}\psi
_{jk}\right\Vert _{L^{\infty }\left( \mathbb{S}^{d}\right) } \\
&=&:I+II\text{ .}
\end{eqnarray*}%
As in \cite{bkmpAoSb}, \cite{dgm}, we have:%
\begin{equation*}
II=O\left( n^{-\frac{r}{2r+d}}\right) \text{ .}
\end{equation*}%
For what concerns $I$, we have instead 
\begin{equation*}
I\leq \sum_{j=0}^{J_{n}}\mathbb{E}\left\Vert \sum_{k=1}^{N_{j}}\left( w_{j;p}%
\widehat{\beta }_{jk}-\beta _{jk}\right) \psi _{jk}\right\Vert _{L^{\infty
}\left( \mathbb{S}^{d}\right) }\leq C\sum_{j=0}^{J_{n}}B^{j}\mathbb{E}\left[
\sup_{k}\left( w_{j;p}\widehat{\beta }_{jk}-\beta _{jk}\right) \right]
\end{equation*}%
\begin{eqnarray*}
&\leq &C\sum_{j=0}^{J_{n}}B^{j}\mathbb{E}\left[ \sup_{k}\left( \widehat{%
\beta }_{jk}-\beta _{jk}\right) \right] I\left( \left\vert
A_{js;p}\right\vert \geq \frac{\kappa n^{-\frac{p}{2}}}{2}\right) \\
&&+C\sum_{j=0}^{J_{n}}B^{j}\mathbb{E}\left[ \sup_{k}\left( \widehat{\beta }%
_{jk}-\beta _{jk}\right) I\left( \left\vert \widehat{A}_{js;p}-A_{js;p}%
\right\vert \geq \frac{\kappa n^{-\frac{p}{2}}}{2}\right) \right] \\
&&+C\sum_{j=0}^{J_{n}}B^{j}\sup_{k}\left\vert \beta _{jk}\right\vert \mathbb{%
E}\left[ I\left( \left\vert \widehat{A}_{js;p}-A_{js;p}\right\vert \geq
\kappa n^{-\frac{p}{2}}\right) \right] \\
&&+C\sum_{j=0}^{J_{n}}B^{j}\sup_{k}\left\vert \beta _{jk}\right\vert I\left(
\left\vert A_{js;p}\right\vert <2\kappa n^{-\frac{p}{2}}\right) \\
&=&Aa+Au+Ua+Uu\text{ .}
\end{eqnarray*}%
Again, we choose $J_{1,n}$ such that 
\begin{equation*}
B^{J_{1,n}}=\kappa ^{\prime }n^{\frac{1}{2r+d}}\text{ ; \ }I\left(
\left\vert A_{js;p}\right\vert \geq \frac{\kappa n^{-\frac{p}{2}}}{2}\right)
=0\text{ for }j>J_{1,n}\text{ ,}
\end{equation*}%
and, similarly to \cite{bkmpAoSb}, \cite{dgm}, we obtain 
\begin{eqnarray*}
Aa &\leq &CJ_{1,n}n^{-\frac{1}{2}}B^{J_{1,n}}\leq Cn^{-\frac{r}{2\left(
r+1\right) }}\text{ ;} \\
Uu &\leq &C\left\{ B^{-J_{1,n}\left( r+1\right) }+B^{-J_{1,n}}\right\} \leq
Cn^{-\frac{r}{2\left( r+1\right) }}\text{ .}
\end{eqnarray*}%
The other two terms $Au$ and $Ua$ are similar to the case previously
described. For general $\pi $ and $q$, we observe that $\mathcal{B}_{\pi
q}^{r}\subset \mathcal{B}_{\infty \infty }^{r^{\prime }}$, $r^{\prime
}=r-2/\pi $. Hence we obtain%
\begin{equation*}
\mathbb{E}\left\Vert \widehat{f}-f\right\Vert _{L^{\infty }\left( \mathbb{S}%
^{d}\right) }\leq CJ_{n}n^{-\frac{r^{\prime }}{2r^{\prime }+d}}=CJ_{n}n^{-%
\frac{r-d/\pi }{2\left( r-d\left( 1/\pi -1/2\right) \right) }}\text{ ,}
\end{equation*}%
as claimed.

\textbf{CASE\ II:} \emph{Sparse Case}

The proof follows the same procedure of the regular case. Indeed, recalling
that we have $\mathcal{B}_{\pi q}^{r}\subset \mathcal{B}_{pq}^{r-d\left( 
\frac{1}{\pi }-\frac{1}{p}\right) }$, we have%
\begin{equation*}
\mathbb{E}\left\Vert \widehat{f}-f\right\Vert _{L^{p}\left( \mathbb{S}%
^{d}\right) }^{p}
\end{equation*}

\begin{eqnarray*}
&\leq &2^{p-1}\mathbb{E}\left\Vert
\sum_{j=0}^{J_{n}}\sum_{s=1}^{S_{j}}\sum_{k\in R_{j;s}}\left( w_{js;p}%
\widehat{\beta }_{jk}-\beta _{jk}\right) \psi _{jk}\right\Vert _{L^{p}\left( 
\mathbb{S}^{d}\right) }^{p}+\left\Vert \sum_{j>J_{n}}\sum_{k=1}^{N_{j}}\beta
_{jk}\psi _{jk}\right\Vert _{L^{p}\left( \mathbb{S}^{d}\right) }^{p} \\
&=&:I+II\text{ .}
\end{eqnarray*}%
Also in this case, as in \cite{bkmpAoSb}, \cite{dgm}, because $r-d\left( 
\frac{1}{\pi }-\frac{1}{p}\right) \geq \left( r-d\left( \frac{1}{\pi }-\frac{%
1}{p}\right) \right) /2\left( r-d\left( \frac{1}{\pi }-\frac{1}{2}\right)
\right) $, we have for the bias term:%
\begin{equation*}
II=O\left( n^{-p\left( r-d\left( \frac{1}{\pi }-\frac{1}{p}\right) \right)
/2\left( r-d\left( \frac{1}{\pi }-\frac{1}{2}\right) \right) }\right) \text{
.}
\end{equation*}%
On the other hand, we split $I$ again into four terms as above. On one hand,
we obtain 
\begin{eqnarray*}
Au &\leq &\sum_{j\leq J_{n}}\sum_{s=1}^{S_{j}}\sum_{k\in R_{j;s}}\left\Vert
\psi _{jk}\right\Vert _{L^{p}\left( \mathbb{S}^{d}\right) }^{p}\left( 
\mathbb{E}\left[ \left\vert \widehat{\beta }_{jk}-\beta _{jk}\right\vert
^{2p}\right] \right) ^{\frac{1}{2}} \\
&&\times \left( \mathbb{P}\left( \left\vert \widehat{A}_{js;p}-A_{js;p}%
\right\vert \geq \frac{\kappa t_{n}^{p}}{2}\right) \right) ^{\frac{1}{2}} \\
Ua &\leq &\sum_{j\leq J_{n}}\sum_{s=1}^{S_{j}}\sum_{k\in R_{j;s}}\left\Vert
\psi _{jk}\right\Vert _{L^{p}\left( \mathbb{S}^{d}\right) }^{p}\left\vert
\beta _{jk}\right\vert ^{p}\left( \mathbb{P}\left( \left\vert \widehat{A}%
_{js;p}-A_{js;p}\right\vert \geq \kappa t_{n}^{p}\right) \right) \text{ ,}
\end{eqnarray*}%
whose upper bounds recall exactly the same procedure developed in regular
zone. On the other hand, consider initially:%
\begin{equation*}
Aa\leq Cn^{-\frac{p}{2}}\sum_{j\leq J_{n}}\sum_{s=1}^{S_{j}}\sum_{k\in
R_{j;s}}B^{jd\left( \frac{p}{2}-1\right) }I\left( \left\vert
A_{js;p}\right\vert \geq \frac{\kappa n^{-\frac{p}{2}}}{2}\right) \text{ .}
\end{equation*}%
In this case, we fix $J_{2n}$ so that 
\begin{equation*}
B^{J_{2n}}=O\left( n^{\frac{1}{2\left( r-d\left( \frac{1}{\pi }-\frac{1}{2}%
\right) \right) }}\right) \text{ , }I\left( \left\vert A_{js;p}\right\vert
\geq \frac{t_{n}^{p}}{2}\right) \equiv 0\text{ for }j\geq J_{2n}\text{ .}
\end{equation*}%
to obtain%
\begin{eqnarray}
Aa &\leq &Cn^{-\frac{p}{2}}\sum_{j\leq J_{2n}}\sum_{s=1}^{S_{j}}\sum_{k\in
R_{j;s}}B^{jd\left( \frac{p}{2}-1\right) }I\left( \left\vert
A_{js;p}\right\vert \geq \frac{t_{n}^{p}}{2}\right)  \label{Aaeqn} \\
&\leq &Cn^{-\frac{p}{2}}\sum_{j\leq J_{2n}}B^{jd\left( \frac{p}{2}-1\right)
}\sum_{s=1}^{S_{j}}\ell _{j}I\left( \left\vert A_{js;p}\right\vert \geq 
\frac{t_{n}^{p}}{2}\right)  \notag \\
&\leq &Cn^{-\frac{p}{2}}t_{n}^{-p}\sum_{j\leq J_{2n}}B^{jd\left( \frac{p}{2}%
-1\right) }B^{-prj}B^{-jdp\left( \frac{1}{2}-\frac{1}{\pi }\right) }  \notag
\\
&\leq &CB^{J_{2n}\left( -p\left( r-d\left( \frac{1}{\pi }-\frac{1}{p}\right)
\right) \right) }  \notag \\
&\leq &Cn^{-\frac{p\left( r-d\left( \frac{1}{\pi }-\frac{1}{p}\right)
\right) }{2\left( r-d\left( \frac{1}{\pi }-\frac{1}{2}\right) \right) }}%
\text{ .}  \notag
\end{eqnarray}%
where we used the inequality%
\begin{equation*}
\sum_{k=1}^{N_{j}}\left\vert \beta _{jk}\right\vert ^{p}\leq \left(
\sum_{k=1}^{N_{j}}\left\vert \beta _{jk}\right\vert ^{\pi }\right) ^{\frac{p%
}{\pi }}\text{ .}
\end{equation*}

Consider now%
\begin{eqnarray}
Uu &\leq &C\sum_{j\leq J_{n}}B^{jd\left( \frac{p}{2}-1\right) }\ell
_{j}\sum_{s=1}^{S_{j}}A_{js;p}I\left( \left\vert A_{js;p}\right\vert
<2t_{n}^{p}\right)  \label{Uu} \\
&=&C\sum_{j\leq J_{2n}}B^{jd\left( \frac{p}{2}-1\right) }\ell
_{j}\sum_{s=1}^{S_{j}}A_{js;p}I\left( \left\vert A_{js;p}\right\vert
<2t_{n}^{p}\right)  \notag \\
&+&\sum_{j=J_{2n}}^{J_{n}}B^{jd\left( \frac{p}{2}-1\right) }\ell
_{j}\sum_{s=1}^{S_{j}}A_{js;p}I\left( \left\vert A_{js;p}\right\vert
<2t_{n}^{p}\right)  \notag \\
&=&Uu_{1}+Uu_{2}\text{ .}  \notag
\end{eqnarray}%
As in \cite{bkmpAoSb}, \cite{dgm}, fix 
\begin{equation*}
m=\frac{dp\left( \frac{1}{2}-\frac{1}{p}\right) }{r-d\left( \frac{1}{\pi }-%
\frac{1}{2}\right) }\,\text{\ ,}
\end{equation*}%
so that%
\begin{eqnarray*}
p-m &=&p\frac{r-d\left( \frac{1}{\pi }-\frac{1}{p}\right) }{r-d\left( \frac{1%
}{\pi }-\frac{1}{2}\right) }>0\text{ ;} \\
m-\pi &=&\frac{\frac{d}{2}p-\pi \left( r+\frac{d}{2}\right) }{r-d\left( 
\frac{1}{\pi }-\frac{1}{2}\right) }>0\text{ .}
\end{eqnarray*}%
Furthermore, consider that the following implication holds 
\begin{equation*}
\left( \left\vert A_{js}\right\vert <2t_{n}^{p}\right) \rightarrow \forall
k\in R_{js},\text{ }\left\vert \beta _{jk}\right\vert ^{p}<2\ell
_{j}t_{n}^{p}\text{ ,}
\end{equation*}%
so that 
\begin{equation*}
\forall k\in R_{js},\text{ }\left\vert \beta _{jk}\right\vert ^{p-\pi
}<\left( 2\ell _{j}\right) ^{\frac{p-\pi }{p}}t_{n}^{p-\pi }\text{ .}
\end{equation*}%
Simple calculations lead to%
\begin{eqnarray*}
Uu_{1} &=&C\sum_{j\leq J_{2n}}B^{jd\left( \frac{p}{2}-1\right)
}\sum_{k=1}^{N_{j}}\left\vert \beta _{jk}\right\vert ^{p}I\left( \left\vert
A_{js}\right\vert <2t_{n}^{p}\right) \\
&\leq &C\sum_{j\leq J_{2n}}B^{jd\left( \frac{p}{2}-1\right)
}t\sum_{k=1}^{N_{j}}\left\vert \beta _{jk}\right\vert ^{\pi }\ell _{j}^{1-%
\frac{\pi }{p}}t_{n}^{p-\pi } \\
&\leq &Cn^{\frac{\pi -p}{2}}\sum_{j\leq J_{2n}}B^{j\eta d\left( 1-\frac{\pi 
}{p}\right) }B^{j\left( p\frac{d}{2}-\pi \left( \frac{d}{2}+r\right) \right)
} \\
&=&O\left( n^{-\frac{p\left( r+d\left( \frac{1}{p}-\frac{1}{\pi }\right)
\right) }{2\left( r+d\left( \frac{1}{2}-\frac{1}{\pi }\right) \right) }%
+\delta }\right)
\end{eqnarray*}

We have to study just the last term%
\begin{equation*}
Uu_{2}=C\sum_{j=J_{2n}}^{J_{n}}B^{j\eta \left( 1-\frac{\pi }{p}\right)
}B^{jd\left( \frac{p}{2}-1\right) }\sum_{k=1}^{N_{j}}\left\vert \beta
_{jk}\right\vert ^{p}I\left( \left\vert A_{js;p}\right\vert
<2t_{n}^{p}\right) \text{ .}
\end{equation*}

Analogously, we have%
\begin{eqnarray*}
Uu_{2} &=&C\sum_{j=J_{2n}}^{J_{n}}B^{jd\left( \frac{p}{2}-1\right)
}\sum_{k=1}^{N_{j}}\left\vert \beta _{jk}\right\vert ^{p}I\left( \left\vert
A_{js;p}\right\vert <2t_{n}^{p}\right) \\
&\leq &C\sum_{j=J_{2n}}^{J_{n}}B^{jd\left( \frac{p}{2}-1\right)
}\sum_{k=1}^{N_{j}}\left\vert \beta _{jk}\right\vert ^{m}t^{p-m}\ell _{j}^{1-%
\frac{m}{p}} \\
&\leq &Ct_{n}^{p-m}\ell _{J_{2n}}^{1-\frac{m}{p}%
}\sum_{j=J_{2n}}^{J_{n}}B^{jd\left( \frac{p}{2}-\frac{m}{2}\right) }\left(
\sum_{k=1}^{N_{j}}\left\vert \beta _{jk}\right\vert ^{m}B^{jd\left( \frac{m}{%
2}-1\right) }\right) \\
&\leq &Ct_{n}^{p-m}\ell _{J_{2n}}^{1-\frac{m}{p}%
}\sum_{j=J_{2n}}^{J_{n}}B^{jd\left( \frac{p}{2}-1\right) }B^{-mj\left(
r-d\left( \frac{1}{\pi }-\frac{1}{m}\right) \right) }
\end{eqnarray*}%
We can easily see that%
\begin{equation*}
\left( p-m\right) -m\left( r-d\left( \frac{1}{\pi }+\frac{1}{m}\right)
\right) =0\text{ .}
\end{equation*}%
Hence%
\begin{equation*}
Uu_{2}\leq C\ell _{J_{2n}}^{1-\frac{m}{p}}t_{n}^{p-n}=O\left( n^{-\frac{%
p\left( r+d\left( \frac{1}{p}-\frac{1}{\pi }\right) \right) }{2\left(
r-d\left( \frac{1}{\pi }-\frac{1}{2}\right) \right) }+\delta }\right) \text{
,}
\end{equation*}%
as claimed.

\section{Conclusions\label{sec:conclu}}

In this final Section we shall compare our results with those obtained by
similar procedures, involving needlets, in \cite{bkmpAoSb} and \cite{dgm}.

While in \cite{bkmpAoSb} the authors established minimax results on density
estimation by using local needlet thresholding (i.e., fixing a threshold for
each coefficients), in \cite{dgm} the authors attain the same minimax
results for the nonparametric regression problem on sections of spin $s$
fiber bundles defined on the sphere, which can be reduced to the scalar case
taking $s=0$ (for more details see \cite{dgm}). In both cases, the
convergence rates for the $L^{p}\left( \mathbb{S}^{d}\right) $-loss
functions assume the form 
\begin{equation*}
\sup_{f\in \mathcal{B}_{\pi q}^{r}\left( M\right) }\mathbb{E}\left\Vert 
\widehat{f}-f\right\Vert _{L^{p}\left( \mathbb{S}^{d}\right) }^{p}\leq
c_{p}\left( \log n\right) ^{p}\left( \frac{n}{\log n}\right) ^{-\alpha
_{1}\left( r,\pi ,p\right) }\text{ ,}
\end{equation*}%
where 
\begin{equation*}
\alpha _{1}\left( r,\pi ,p\right) =\left\{ 
\begin{array}{c}
\frac{rp}{2r+d}\text{ \ \ \ \ \ \ \ \ \ \ \ \ for }\pi \geq \frac{dp}{2r+d}
\\ 
\frac{p\left( r-d\left( \frac{1}{\pi }-\frac{1}{p}\right) \right) }{2\left(
r-d\left( \frac{1}{\pi }-\frac{1}{2}\right) \right) }\text{ for }\pi <\frac{%
dp}{2r+d}%
\end{array}%
\right. \text{ .}
\end{equation*}%
In the regular zone, the block thresholding rate we established is faster,
indeed the ratio with the local one is provided by 
\begin{equation*}
\frac{\left( \frac{n}{\log n}\right) ^{-\alpha _{1}\left( r,\pi ,p\right) }}{%
n^{-\alpha \left( r,\pi ,p\right) }}=O\left( \left( \log n\right) ^{\frac{rp%
}{2r+d}}\right) \text{ ;}
\end{equation*}%
on the other hand, in the sparse zone, we obtain worse results, because 
\begin{equation*}
\frac{\left( \frac{n}{\log n}\right) ^{-\alpha _{1}\left( r,\pi ,p\right) }}{%
n^{-\alpha \left( r,\pi ,p\right) }}=O\left( n^{-\delta }\left( \log
n\right) ^{\frac{p\left( r-d\left( \frac{1}{\pi }-\frac{1}{p}\right) \right) 
}{2\left( r-d\left( \frac{1}{\pi }-\frac{1}{2}\right) \right) }}\right) 
\text{ .}
\end{equation*}

This can be motivated by choice of the sample scaling factor $t_{n}$, fixed
to allow optimality in the regular zone. In the sparse zone this is not
possible also in view of the result in Lemma \ref{mainlemma}, where (\ref%
{P_A}) is proportional to $n^{-\gamma }$ and can not be improved. We indeed
recall that in \cite{bkmpAoSb} and in \cite{dgm} the corresponding
inequality, related just on a coefficient instead of a sum of them, follows
Bernstein inequality and, therefore, that probability decays as a negative
exponential. As already mentioned in the Introduction, the best performance
achieved by block thresholding in the regular zone can be explained by the
better trade-off between bias and variance. The latter is due to the
information provided by nearby coefficients, which allows the balance
between variance and bias to be "adaptively smoothed" along the curve, given
a suitable choice of the threshold $t_{n}$. On the other hand, the worse
results obtained in the sparse regions are due to the balance between the
choice of the size of threshold $t_{n}$ and the size of the block. Indeed,
given $t_{n}$, the probability inequality (\ref{P_A}) is suitable to attain
minimax rate in the regular zone if we choose blocks as described in (\ref%
{blokke}). Fixing a smaller size, as for instance $\left( \log N_{j}\right)
^{\gamma }$ (see \cite{hkp}), the convergence rate in the regular zone is
worsened. Our suggestion is to fix the block sizes which ensure the minimax
results in the regular zone; as explained in Remark \ref{remark}, this
warrants optimality in the most relevant case for practitioners, e.g., the
case of a quadratic loss function.

\begin{acknowledgement}
The author thanks Domenico Marinucci for useful discussions.
\end{acknowledgement}

\end{document}